\documentclass[12pt]{amsart}
\usepackage{import}
\usepackage{amssymb}
\usepackage{bm}
\usepackage{amsthm}
\usepackage{graphicx}
\usepackage{psfrag}
\usepackage[usenames,dvipsnames]{xcolor}
\usepackage[labelformat=simple]{subcaption}
\usepackage{enumerate}
\usepackage{url}

\usepackage{algpseudocode}

\usepackage{rotating}

\usepackage{mathdots}
\usepackage{mathtools}
\usepackage{multirow}
\usepackage{blkarray}
\usepackage{relsize}

\usepackage{xy}
\input xy
\xyoption{all}

\usepackage{tikz}
\usetikzlibrary{positioning}
\usetikzlibrary{graphs,graphs.standard}

\newcommand{\excise}[1]{}

\newtheorem{thm}{Theorem}[section]
\newtheorem{lemma}[thm]{Lemma}

\newtheorem{cor}[thm]{Corollary}
\newtheorem{prop}[thm]{Proposition}

\newtheorem{question}[thm]{Question}

\theoremstyle{definition}

\newtheorem{example}[thm]{Example}
\newtheorem{remark}[thm]{Remark}
\newtheorem{defn}[thm]{Definition}

\newtheorem{notation}[thm]{Notation}

\numberwithin{equation}{section}

\newcommand{\ring}[1]{\ensuremath{\mathbb{#1}}}

\renewcommand\>{\rangle}
\newcommand\<{\langle}

\newcommand\NN{\ring{N}}

\newcommand\QQ{\ring{Q}}
\newcommand\RR{\ring{R}}

\newcommand\ZZ{\ring{Z}}

\newcommand\til{\mathord\sim}

\def\ol#1{{\overline {#1}}}
\def\wh#1{{\widehat {#1}}}

\DeclareMathOperator\Ap{Ap} 
\DeclareMathOperator\supp{supp} 

\DeclareMathOperator\rk{rk}

\newcommand{\mathbigger}[1]{\mathlarger{\mathlarger{#1}}}
\begin{document}

\mbox{}
\title[Numerical semigroups, polyhedra, and posets III]{Numerical semigroups, polyhedra, and posets III:\ minimal presentations and face dimension}

\author[T.~Gomes]{Tara Gomes}
\address{School of Mathematics\\University of Minnesota\\Minneapolis, MN 55455}
\email{gomes072@umn.edu}

\author[C.~O'Neill]{Christopher O'Neill}
\address{Mathematics and Statistics Department\\San Diego State University\\San Diego, CA 92182}
\email{cdoneill@sdsu.edu}

\author[E.~Torres Davila]{Eduardo Torres Davila}
\address{School of Mathematics\\University of Minnesota\\Minneapolis, MN 55455}
\email{torre680@umn.edu}

\date{\today}

\begin{abstract}
This paper is the third in a series of manuscripts that examine the combinatorics of the Kunz polyhedron $P_m$, whose positive integer points are in bijection with numerical semigroups (cofinite subsemigroups of $\mathbb Z_{\ge 0}$) whose smallest positive element is $m$.  The faces of $P_m$ are indexed by a family of finite posets (called Kunz posets) obtained from the divisibility posets of the numerical semigroups lying on a given face.  In this paper, we characterize to what extent the minimal presentation of a numerical semigroup can be recovered from its Kunz poset.  In doing so, we prove that all numerical semigroups lying on the interior of a given face of $P_m$ have identical minimal presentation cardinality, and we provide a combinatorial method of obtaining the dimension of a face from its corresponding Kunz poset.  
\end{abstract}

\maketitle


\section{Introduction}
\label{sec:intro}

A \emph{numerical semigroup} is a cofinite subset $S \subseteq \NN$ of the non-negative integers that is closed under addition and contains $0$.  Numerical semigroups are often specified using a set of generators $n_0 < \cdots < n_k$, i.e.,
\[
S = \<n_0, \ldots, n_k\> = \{a_1n_1 + \cdots + a_kn_k : a_i \in \NN\}.
\]
The \emph{Ap\'ery set} of a nonzero element $m \in S$ is the set
\[
\Ap(S;m) = \{n \in S : n - m \notin S\}
\]
of minimal elements of $S$ within each equivalence class modulo $m$.  Since $S$ is cofinite, we~are guaranteed $|\!\Ap(S;m)| = m$, and that $\Ap(S;m)$ contains exactly one element in each equivalence class modulo $m$. The elements of $\Ap(S;m)$ are partially ordered by divisibility, that is, $a \preceq a'$ whenever $a' - a \in \Ap(S;m)$; we call this the \emph{Ap\'ery poset} of~$S$ when $m = n_0$ (the \emph{multiplicity} of $S$).

Numerous recent papers have centered around a family of rational polyhedra whose integer points are in bijection with certain numerical semigroups~\cite{alhajjarkunz, wilfmultiplicity,kaplanwilfconj,kunz,kunznew,kunzcoords}.  For each $m \ge 2$, the \emph{Kunz polyhedron} $P_m$ is a pointed rational cone, translated from the origin, whose positive integer points are in bijection with the numerical semigroups of multiplicity $m$ (we defer precise definitions to Section~\ref{sec:background}).  
This manuscript is the third in a series examining a combinatorial description of the faces of $P_m$~\cite{kunzfaces2,kunzfaces1}.  Given a numerical semigroup $S$ with multiplicity $m$, the \emph{Kunz poset} of $S$ is the partially ordered set with ground set $\ZZ_m$ obtained by replacing each element $a$ of the Ap\'ery poset $\Ap(S;m)$ with its equivalence class $\ol a$ in $\ZZ_m$.  In~\cite{kunzfaces1}, it was shown that the faces of $P_m$ are indexed by a family of posets that, for a given face $F$, coincides with the Kunz poset of any numerical semigroup corresponding to an integer point interior to~$F$.  

One of the primary ways of studying a numerical semigroup $S$ is via a minimal presentation $\rho \subset \NN^{k+1} \times \NN^{k+1}$, each element of which is a pair of factorizations that represents a minimal \emph{relation} or \emph{trade} between the minimal generators of $S$.  In this paper, we begin by introducing, in Section~\ref{sec:kunzminpres}, a \emph{minimal presentation of a Kunz poset}, and subsequently arrive at two main results:
\begin{itemize}
\item 
we obtain a combinatorial solution to the problem of computing face dimension in $P_m$ (Section~\ref{sec:facedim}); and

\item
we prove the cardinality of a minimal presentation of $S$ depends only on its Kunz poset, and thus is fixed within each face of the Kunz polyhedron (Section~\ref{sec:minpressize}).  

\end{itemize}
One consequence of the second item above is a new algorithm for computing a minimal presentation of a numerical semigroup, outlined in Remark~\ref{r:minpresalgorithm}.  We also introduce a Sage package \texttt{KunzPoset} for interfacing between the faces of the Kunz polyhedra, their associated Kunz posets, and the numerical semigroups they contain; see Section~\ref{sec:sagepackage} for an overview of its functionality.

\section{Background}
\label{sec:background}

In the first half of this section, we recall the definition of, and several well-known structural results concerning, minimal presentations of numerical semigroups (for a thorough treatment, see \cite[Chapter~9]{numerical}).  In the second half, we recall basic definitions from polyhedral geometry (see~\cite{ziegler} for a thorough introduction) and define the Kunz polyhedron $P_m$ and a related polyhedron from~\cite{kunzfaces1}.

Throughout this paper, all semigroups are assumed to be commutative and reduced.  

\begin{defn}\label{d:minimalpresentations}
Fix a finitely generated semigroup $S = \<n_0, \ldots, n_k\>$, written additively.  
A \emph{factorization} of an element $n \in S$ is an expression 
$$n = z_0n_0 + \cdots + z_kn_k$$
of $n$ as a sum of generators of $S$, and the \emph{set of factorizations} of $n$ is the set
$$\mathsf Z_S(n) = \{z \in \NN^{k+1} : n = z_0n_0 + \cdots + z_kn_k\}$$
viewed as a subset of $\NN^{k+1}$.  The \emph{factorization homomorphism} of $S$ is the function $\varphi_S:\NN^{k+1} \to S$ given by
$$\varphi_S(z_0, \ldots, z_k) =  z_0n_0 + \cdots + z_kn_k$$
sending each $(k+1)$-tuple to the element of $S$ it is a factorization of.  The \emph{kernel} of $\varphi_S$ is the equivalence relation $\til = \ker \varphi_S$ that sets $z \sim z'$ whenever $\varphi_S(z) = \varphi_S(z')$.  Each such relation $z \sim z'$ is called a \emph{trade} of $S$.  The kernel $\til$ is in fact a \emph{congruence}, meaning that $z \sim z'$ implies $(z + z'') \sim (z' + z'')$ for all $z, z', z'' \in \NN^{k+1}$; see~\cite{congruencesurvey,fingenmon} for a survey on the role of congruences in this context.  A~subset $\rho \subset \ker\varphi_S$, viewed as a subset of $\NN^{k+1} \times \NN^{k+1}$, is a \emph{presentation} of $S$ if $\ker\varphi_S$ is the smallest congruence on $\NN^{k+1}$ containing $\rho$; see \cite[Propsition~8.4]{numerical} for a thorough description of the smallest congruence containing a given set of relations.  We~say $\rho$ is a \emph{minimal presentation} if it is minimal with respect to containment among all presentations for $S$.  
\end{defn}

It is known that every finitely generated semigroup $S$ has a finite minimal presentation, and that all minimal presentations of $S$ have equal cardinality.  Theorem~\ref{t:allminimalpresentations} provides a combinatorial characterization of the minimal presentations of $S$ in the case where $S$ is cancellative, a consequence of \cite[Lemma~2.1]{deltabetti} also surveyed for numerical semigroups in~\cite{numericalappl}.  

\begin{defn}\label{d:factorizationgraph}
Fix a subset $Z \subset \NN^d$.  The \emph{factorization graph} of $Z$ is the graph $\nabla_Z$ whose vertices are the elements of $Z$ in which two tuples $z, z' \in Z$ are connected by an edge if $z_i > 0$ and $z_i' > 0$ for some $i$.  If $S$ is a finitely generated semigroup and $n \in S$, then we write $\nabla_n = \nabla_Z$ for $Z = \mathsf Z_S(n)$.  
\end{defn}

\begin{thm}\label{t:allminimalpresentations}
Fix a finitely generated cancellative semigroup $S = \<n_0, \ldots, n_k\>$.  
A~set $\rho \subset \ker \varphi_S$ is a presentation of $S$ if and only if for every $n \in S$, a connected graph is obtained from $\nabla_n$ by adding an edge for each pair of factorizations of $n$ in $\rho$.  Furthermore, a presentation $\rho$ is minimal if and only if for every $n \in S$, the number of connected components in $\nabla_n$ is one more than the number of relations in $\rho$ containing factorizations of $n$.  
\end{thm}

A \emph{rational polyhedron} $P \subset \RR^d$ is the set of solutions to a finite list of linear inequalities with rational coefficients, that is,
$$
P = \{x \in \RR^d : Ax \le b\}
$$
for some rational matrix $A$ and vector $b$. If none of the inequalities can be omitted without altering $P$, we call this list the \emph{$H$-description} or \emph{facet description} of $P$ (such a list of inequalities is unique up to reordering and scaling by positive constants).  The~inequalities appearing in the H-description of $P$ are called \emph{facet inequalities} of~$P$.

Given a facet inequality $a_1x_1 + \cdots + a_dx_d \le b$ of $P$, the intersection of $P$ with the equation $a_1x_1 + \cdots + a_dx_d = b$ is called a \emph{facet} of $P$. A \emph{face} $F$ of $P$ is a subset of $P$ equal to the intersection of some collection facets of $P$.  The set of facets containing $F$ is called the \emph{H-description} or \emph{facet description} of $F$.  The \emph{dimension} of a face $F$ is the dimension $\dim(F)$ of the affine linear span of $F$.  The \emph{relative interior} of a face $F$ is the set of points in $F$ that do not also lie in a face of dimension strictly smaller than $F$ (or, equivalently, do not lie in a proper face of $F$).  We say $F$ is a \emph{vertex} if $\dim(F) = 0$, and \emph{edge} if $\dim(F) = 1$ and $F$ is bounded, a \emph{ray} if $\dim(F) = 1$ and $F$ is unbounded, and a \emph{ridge} if $\dim(F) = d - 2$.

If the origin is the unique point lying in the intersection of all facets of $P$ (or, equivalently, if $b = 0$ in the H-description of $P$), then we call $P$ a \emph{pointed cone}.
Separately, we say $P$ is a \emph{polytope} if $P$ is bounded.  If $P$ is a pointed cone, then any face $F$ equals the non-negative span of the rays of $P$ it contains, and if~$P$ is a polytope, then any face $F$ equals the convex hull of the set of vertices of $P$ it contains; in each case, we call this the \emph{V-description} of $F$.

A \emph{partially ordered set} (or \emph{poset}) is a set $Q$ equipped with a \emph{partial order} $\preceq$, that is, a reflexive, antisymmetric, and transitive relation.  We say $q$ \emph{covers} $q'$ if $q' \prec q$ and there is no intermediate element $q''$ with $q' \prec q'' \prec q$.  If $(Q, \preceq)$ has a unique minimal element $0 \in Q$, the \emph{atoms} of $Q$ are the elements that cover $0$.  The set of faces of a polyhedron $P$ forms a poset under containment that is a \emph{lattice} (i.e., every element has a unique greatest common divisor and least common multiple) and is \emph{graded}, where the height function is given by dimension.  If $P$ is a cone, then every face of $P$ equals the sum of some collection of extremal rays and the intersection of some collection of facets, meaning the face lattice of $P$ is both \emph{atomic} and \emph{coatomic}.

\begin{defn}\label{d:kunzcoords}
Fix $m \in \ZZ_{\ge 2}$, and a numerical semigroup $S$ containing $m$. Write
$$
\Ap(S;m) = \{0, a_1, \ldots, a_{m-1}\},
$$
where $a_i = mx_i + i$ for each $i =
1, \ldots, m-1$.  We refer to the tuples $(a_1, \ldots, a_{m-1})$ and $(x_1,
\ldots, x_{m-1})$ as the \emph{Ap\'ery tuple}/\emph{Ap\'ery coordinates} and the
\emph{Kunz tuple}/\emph{Kunz coordinates} of $S$, respectively. \end{defn}

\begin{defn}\label{d:kunzandgroupcone}
Fix a finite Abelian group $G$, and let $m = |G|$. The \emph{group cone}
$\mathcal C(G) \subset \RR_{\ge 0}^{m-1}$ is the pointed cone with facet inequalities
$$
\begin{array}{r@{}c@{}l@{\qquad}l}
x_i + x_j &{}\ge{}& x_{i+j} & \text{for } i, j \in G \setminus \{0\} \text{ with } i + j \ne 0,
\end{array}
$$
where the coordinates of $\RR^{m-1}$ are indexed by $G \setminus \{0\}$.
Additionally, for each integer $m \ge 2$, let $P_m$ denote the translation of
$\mathcal C(\ZZ_m)$ with vertex $(-\tfrac{1}{m}, \ldots, -\tfrac{m-1}{m})$,
whose facets are given by
$$
\begin{array}{r@{}c@{}l@{\qquad}l}
x_i + x_j &{}\ge{}& x_{i+j} & \text{for } 1 \le i \le j \le m - 1 \text{ with } i + j < m, \text{ and} \\
x_i + x_j + 1 &{}\ge{}& x_{i+j-m} & \text{for } 1 \le i \le j \le m - 1 \text{ with } i + j > m.
\end{array}
$$
We refer to $P_m$ as the \emph{Kunz polyhedron}.
\end{defn}

Parts~(a) and~(b) of the following theorem appear in \cite{kunz} and \cite{kunzfaces1}, respectively.

\begin{thm}\label{t:kunzlatticepts}
Fix an integer $m \ge 2$.
\begin{enumerate}[(a)]
\item
The set of all Kunz tuples of numerical semigroups containing $m$ coincides with
the set of integer points in $P_m$.

\item
The set of all Ap\'ery tuples of numerical semigroups containing $m$ coincides
with the set of integer points $(a_1, \ldots, a_{m-1})$ in $\mathcal C(\ZZ_m)$
with $a_i \equiv i \bmod m$ for every $i$.

\end{enumerate}
\end{thm}

In view of Theorem~\ref{t:kunzlatticepts}, given a face $F \subset \mathcal C(\ZZ_m)$, we say $F$ \emph{contains} a numerical semigroup $S$ if the Ap\'ery tuple of $S$ lies in the relative interior of $F$.  Analogously, we say a face $F'
\subset P_m$ \emph{contains} $S$ if the Kunz tuple of $S$ lies in the relative interior of $F'$.

\begin{thm}[{\cite[Theorem~3.4]{kunzfaces1}}]\label{t:groupconefacelattice}
Fix a finite Abelian group $G$ and a face $F \subset \mathcal C(G)$.
\begin{enumerate}[(a)]
\item
The set
$H = \{h \in G : x_g = 0 \text{ for all } x \in F\}$
is a subgroup of $G$ (called the \emph{Kunz subgroup} of $F$), and the relation
$P = (G/H, \preceq)$, with unique minimal element $\ol 0$ and $\ol a \preceq_P
\ol b$ for distinct $a, b \in G \setminus H$ whenever $x_a + x_{b-a} = x_b$ for all $x \in F$, is a well-defined partial order (called the \emph{Kunz poset} of $F$).

\item
If $G = \ZZ_m$ with $m \ge 2$ and $F$ contains a numerical semigroup $S$, then
the Kunz subgroup of $F$ is trivial and the Kunz poset of $F$ equals the Kunz
poset of $S$.

\item
In the Kunz poset $P$ of $F$, $\ol b$ covers $\ol a$ if and only if $\ol b - \ol
a$ is an atom of $P$.

\end{enumerate}
\end{thm}

Following Theorem~\ref{t:groupconefacelattice}(a), for the purposes of this manuscript, a poset $P$ is called a \emph{Kunz poset} if it equals the Kunz poset of a face of some group cone $\mathcal C(G)$.  

\begin{example}\label{e:groupconefacelattice}
The numerical semigroups $S = \<6,7,8,9\>$ and $S' = \<6,19,26,33\>$ both have Kunz poset depicted in Figure~\ref{f:findminpres1}(a).  Their Ap\'ery sets 
\[
\Ap(S; 6) = \{0,7,8,9,16,17\}
\qquad \text{and} \qquad
\Ap(S';6) = \{0,19,26,33,52,69\},
\]
when written $\{0, a_1, \ldots, a_5\}$ with $a_i \equiv i \bmod 6$, satisfy the equalities
\[
a_2 + a_3 = a_5.
\qquad
a_1 + a_3 = a_4,
\qquad \text{and} \qquad
2a_2 = a_4
\]
from Definition~\ref{d:kunzandgroupcone}.  These equalities indicate which 3 facets of $\mathcal C(\ZZ_6)$ contain $(a_1, \ldots, a_5)$, 
and the resulting relations in the Kunz poset are 
\[
2 \preceq 5, \,
3 \preceq 5,
\qquad
1 \preceq 4, \,
3 \preceq 4,
\qquad \text{and} \qquad
2 \preceq 4,
\]
respectively, along with $0 \preceq i$ for each $i \in \ZZ_6$.  

On the other hand, one can readily check that $(9,6,3,4,5,6,3) \in \mathcal C(\ZZ_8)$, lying in the face $F \subset \mathcal C(\ZZ_8)$ with facet equalities
\[
x_3 + x_7 = x_2,
\quad
2x_3 = x_6,
\quad
2x_7 = x_6,
\quad
x_3 + x_6 = x_1,
\quad
x_2 + x_7 = x_1,
\quad \text{and} \quad
x_4 + x_5 = x_1, 
\]
and whose corresponding Kunz poset is depicted in Figure~\ref{f:findminpres1}(b).  No numerical semigroup can lie in $F$ since the second and third equalities above force $a_3 = a_7$ to hold for all $(a_1, \ldots, a_7) \in F \cap \ZZ^7$, so $a_3 \equiv 3 \bmod 8$ and $a_7 \equiv 7 \bmod 8$ cannot both hold.  
\end{example}

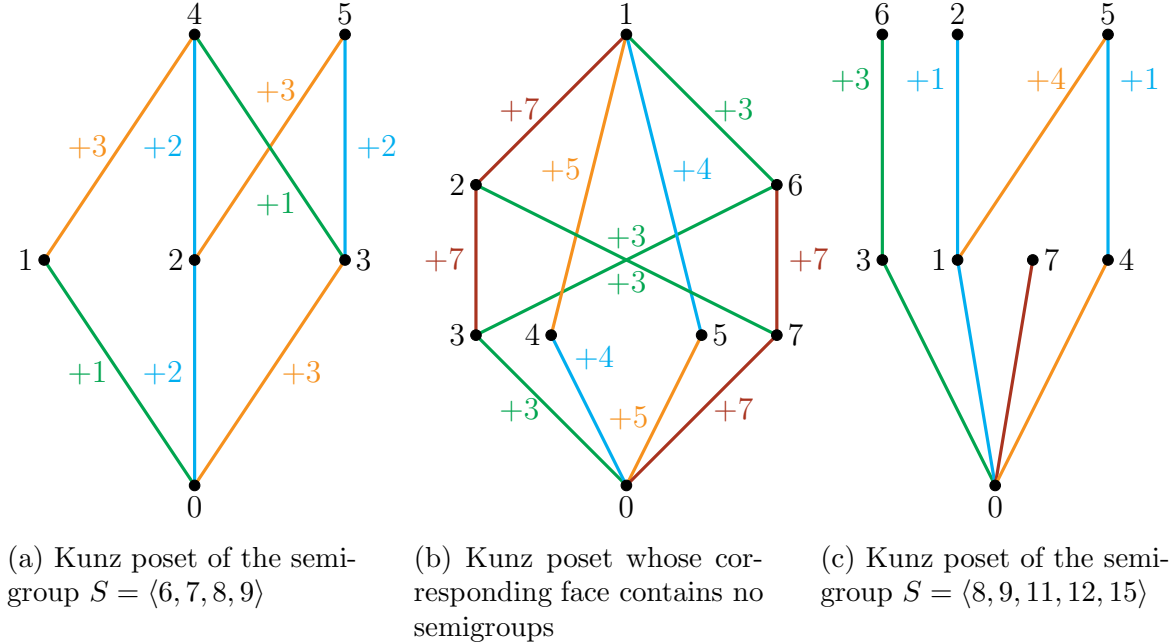
\begin{figure}[t]
\centering
\begin{subfigure}[b]{.3\linewidth}
\centering
\begin{tikzpicture}
\coordinate (0) at (2,0);
\coordinate (1) at (0,3);
\coordinate (2) at (2,3);
\coordinate (3) at (4,3);
\coordinate (4) at (2,6);
\coordinate (5) at (4,6);

\draw[very thick,Green] (0) -- (1) node[midway,left] {$+1$};
\draw[very thick,Cyan] (0) -- (2) node[midway,left] {$+2$};
\draw[very thick,BurntOrange] (0) -- (3) node[midway,right] {$+3$};

\draw[very thick,BurntOrange] (1) -- (4) node[midway,left] {$+3$};
\draw[very thick,Cyan] (2) -- (4) node[midway,left] {$+2$};
\draw[very thick,BurntOrange] (2) -- (5) node[near end,left] {$+3$};
\draw[very thick,Green] (3) -- (4) node[near start,left] {$+1$};
\draw[very thick,Cyan] (3) -- (5) node[midway,right] {$+2$};

\filldraw (0) circle (2pt) node[anchor=north] {$0$};

\filldraw (1) circle (2pt) node[anchor=east] {$1$};
\filldraw (2) circle (2pt) node[anchor=east] {$2$};
\filldraw (3) circle (2pt) node[anchor=west] {$3$};

\filldraw (4) circle (2pt) node[anchor=south] {$4$};
\filldraw (5) circle (2pt) node[anchor=south] {$5$};
\end{tikzpicture}
\caption{Kunz poset of the semigroup $S = \<6, 7, 8, 9\>$
\\
\text{ }}
\label{sf:detailedexample}
\end{subfigure}%
\hspace{0.04\linewidth}
\begin{subfigure}[b]{.3\linewidth}
\centering
\begin{tikzpicture}
\coordinate (0) at (2,0);
\coordinate (1) at (2,6);
\coordinate (2) at (0,4);
\coordinate (3) at (0,2);
\coordinate (4) at (1,2);
\coordinate (5) at (3,2);
\coordinate (6) at (4,4);
\coordinate (7) at (4,2);

\draw[very thick,Green] (0) -- (3) node[midway, left] {$+3$};
\draw[very thick,Cyan] (0) -- (4) node[pos=0.85, right] {$+4$};
\draw[very thick,BurntOrange] (0) -- (5) node[pos=0.45, left] {$+5$};
\draw[very thick,Mahogany] (0) -- (7) node[midway, right] {$+7$};

\draw[very thick,Mahogany] (3) -- (2) node[midway, left] {$+7$};
\draw[very thick,Green] (3) -- (6) node[midway, below] {$+3$};
\draw[very thick,BurntOrange] (4) -- (1) node[pos=0.55, left] {$+5$};
\draw[very thick,Cyan] (5) -- (1) node[pos=0.55, right] {$+4$};
\draw[very thick,Green] (7) -- (2) node[midway, above] {$+3$};
\draw[very thick,Mahogany] (7) -- (6) node[midway, right] {$+7$};

\draw[very thick,Green] (6) -- (1) node[midway, right] {$+3$};
\draw[very thick,Mahogany] (2) -- (1) node[midway, left] {$+7$};

\filldraw (0) circle (2pt) node[anchor=north] {$0$};

\filldraw (3) circle (2pt) node[anchor=east] {$3$};
\filldraw (2) circle (2pt) node[anchor=east] {$2$};
\filldraw (5) circle (2pt) node[anchor=west] {$5$};

\filldraw (1) circle (2pt) node[anchor=south] {$1$};
\filldraw (4) circle (2pt) node[anchor=east] {$4$};
\filldraw (7) circle (2pt) node[anchor=west] {$7$};

\filldraw (6) circle (2pt) node[anchor=west] {$6$};
\end{tikzpicture}
\caption{Kunz poset whose corresponding face contains no semigroups}
\label{sf:nosemigroupsexample}
\end{subfigure}%
\hspace{0.04\linewidth}
\begin{subfigure}[b]{.3\linewidth}
\centering
\begin{tikzpicture}
\coordinate (0) at (2.5,0);
\coordinate (1) at (2,3);
\coordinate (2) at (2,6);
\coordinate (3) at (1,3);
\coordinate (4) at (4,3);
\coordinate (5) at (4,6);
\coordinate (6) at (1,6);
\coordinate (7) at (3,3);

\draw[very thick, Green] (0) -- (3);
\draw[very thick, Cyan] (0) -- (1);
\draw[very thick, Mahogany] (0) -- (7);
\draw[very thick, BurntOrange] (0) -- (4);

\draw[very thick, Green] (3) -- (6) node[pos=0.8, left] {$+3$};
\draw[very thick, Cyan] (1) -- (2) node[pos=0.8, left] {$+1$};
\draw[very thick, BurntOrange] (1) -- (5) node[pos=0.8, left] {$+4$};
\draw[very thick, Cyan] (4) -- (5) node[pos=0.8, right] {$+1$};

\filldraw (0) circle (2pt) node[anchor=north] {$0$};

\filldraw (1) circle (2pt) node[anchor=east] {$1$};
\filldraw (3) circle (2pt) node[anchor=east] {$3$};
\filldraw (4) circle (2pt) node[anchor=west] {$4$};
\filldraw (7) circle (2pt) node[anchor=west] {$7$};

\filldraw (2) circle (2pt) node[anchor=south] {$2$};
\filldraw (6) circle (2pt) node[anchor=south] {$6$};
\filldraw (5) circle (2pt) node[anchor=south] {$5$};
\end{tikzpicture}
\caption{Kunz poset of the semigroup $S = \<8, 9, 11, 12, 15\>$
\\
\text{ }}
\label{sf:bettidimension}
\end{subfigure}%
\caption{Hasse diagrams of posets from Examples~\ref{e:groupconefacelattice},~\ref{e:minpres}, and~\ref{e:facedimcomparison}.}
\label{f:findminpres1}
\end{figure}

\section{Minimal presentations of Kunz posets}
\label{sec:kunzminpres}

In this section, we demonstrate that Kunz posets inherit a natural additive structure from the face of $\mathcal C(G)$ they correspond to (Theorem~\ref{t:kunznilsemigroup}), and use this to develop the notion of minimal presentation of a Kunz poset.  

\begin{defn}\label{d:nilsemigroup}
Fix a semigroup $(N, +)$.  A \emph{nil} is an element $\infty \in N$ such that $a + \infty = \infty$ for all $a \in N$.  We say an element $a \in N$ is \emph{nilpotent} if $n a = \infty$ for some $n \in \ZZ_{\ge 1}$, and \emph{partly cancellative} if $a + b = a + c \ne \infty$ implies $b = c$ for all $b, c \in N$.  We~say $N$ is a \emph{nilsemigroup with identity} (or just a \emph{nilsemigroup}) if $N$ has an identity element and every non-identity element is nilpotent, and that $N$ is \emph{partly cancellative} if every non-nil element of $N$ is partly cancellative.  
\end{defn}

\begin{remark}\label{r:nilsemigroup}
Definition~\ref{d:nilsemigroup} is somewhat nonstandard; the usual definition of nilsemigroup requires every element to be nilpotent, preventing the existence of an identity element, while a \emph{nilmonoid} has both a nil and an identity.  For consistency with the other papers in this series, we prefer the term nilsemigroup.  
\end{remark}

\begin{thm}\label{t:kunznilsemigroup}
Fix a face $F \subset \mathcal C(G)$ with Kunz subgroup $H$.  Define a commutative operation $\oplus$ on the set $N = (G/H) \cup \{\infty\}$ so that $\infty$ is nil and for all $a, b \in G$,
\[
\ol a \oplus \ol b = \begin{cases}
\ol a + \ol b & \text{if $x_a + x_b = x_{a+b}$ for all $x \in F$;} \\
\infty & \text{otherwise.}
\end{cases}
\]
Under this operation, $(N, \oplus)$ is a partly cancellative nilsemigroup (which we call the \emph{Kunz nilsemigroup} of $F$).  Moreover, the divisibility poset of $N \setminus \{\infty\}$ equals the Kunz poset $P = (G/H, \preceq)$ of $F$.  
\end{thm}

\begin{proof}
By Theorem~\ref{t:groupconefacelattice}(a), $\oplus$ is well-defined, and by \cite[Corollary~3.7]{kunzfaces1}, we may assume $H = \{0\}$ and $G/H \cong G$.  By definition, $0$ and $\infty$ are the identity and nil, respectively.  To check associativity of $\oplus$, by the associativity of $G$ it suffices to argue $(a \oplus b) \oplus c = \infty$ if and only if $a \oplus (b \oplus c) = \infty$.  If $(a \oplus b) \oplus c \ne \infty$, then $(a \oplus b) \oplus c = a + b + c$, so
\[
x_a + x_b + x_c = x_{a+b} + x_c = x_{a+b+c} \ge x_a + x_{b+c} \ge x_a + x_b + x_c
\]
for all $x \in F$.  This means $b \oplus c = b + c$, and thus $a \oplus (b \oplus c) = a + b + c$ as well.  As the converse direction follows symmetrically, this proves $N$ is a nilsemigroup.  There are two remaining claims to verify:\ (i) partial cancellativity follows since $a \oplus b = a \oplus c \ne \infty$ implies $a + b = a + c$ and thus $b = c$; and (ii) $P$ is the divisibility poset of $N \setminus \{\infty\}$ since  for $a, b \in N \setminus \{\infty\}$, both $a \preceq_P b$ and $a \oplus (b - a) = b$ are equivalent to the statement $x_a + x_{b-a} = x_b$ for all $x \in F$.  This completes the proof.  
\end{proof}

As a consequence of Theorem~\ref{t:kunznilsemigroup}, Kunz posets can be seen to inheret the additive and factorization structure of their corresponding Kunz nilsemigroup.  Allusions to this idea could be seen in Section~3 of \cite{kunzfaces1} (e.g., in that cover relations of $P$ are naturally labeled by atoms of $P$), but Theorem~\ref{t:kunznilsemigroup} ensures that a full additive nilsemigroup structure can indeed be recovered from the corresponding face $F$.  As~such, we often invoke the underlying nilsemigroup structure of $N = P \cup \{\infty\}$ when working with elements of a Kunz poset $P$.  We now make this formal.  

In the remainder of this section, unless stated otherwise, assume $P$ is a Kunz poset on $\ZZ_m$
with atom set $\mathcal A(P) = \{p_1, \ldots, p_k\}$, and $N$ is the corresponding Kunz nilsemigroup.  

\begin{defn}\label{d:posetfactorizations}
We define $\mathsf Z_P(p) = \mathsf Z_N(p)$ for $p \in P$, and
$$\mathsf Z_P(\infty) = \mathsf Z_N(\infty) = \NN^k \setminus \bigcup_{p \in P} \mathsf Z_P(p).$$
Moreover,  the \emph{factorization homomorphism} of $P$ is the map
$$\varphi_P:\NN^k \to P \cup \{\infty\},$$
sending each $k$-tuple~$z$ to the element $p \in P \cup \{\infty\}$ with $z \in \mathsf Z_P(p)$.  Note that this coincides with the factorization homomorphism of $N$, that is, 
$$\varphi_P(z) = \varphi_N(z) = z_1p_1 + \cdots + z_kp_k,$$
for $z \notin \mathsf Z_P(\infty)$, 
where $P$ and $N$ both have atom set $\mathcal A(P) = \{p_1, \ldots, p_k\}$.  
\end{defn}


\begin{remark}\label{r:chains}
Factorizations of a Kunz poset element $p \in P$ can be viewed in the Hasse diagram in terms of chains from $0$ to $p$.  
See Figure~\ref{f:rearrangechains} for a depiction.  Notice that different chains can correspond to the same factorization, since factorizations are inherently unordered sums of atoms.  
\end{remark}

\begin{figure}[t]
\centering
\begin{subfigure}[b]{.3\linewidth}
\centering
\begin{tikzpicture}
\draw[ultra thick, red] (1,0) -- (0,1);
\draw[thick] (1,0) -- (2,1);

\draw[ultra thick, red] (0,1) -- (1,2);
\draw[thick] (2,1) -- (1,2);
\draw[thick] (2,1) -- (3,2);

\draw[ultra thick, red] (1,2) -- (2,3);
\draw[thick] (3,2) -- (2,3);

\filldraw (1,0) circle (2pt) node[anchor=north] {$0$};

\filldraw (0,1) circle (2pt) node[anchor=north] {$3$};
\filldraw (2,1) circle (2pt) node[anchor=north] {$2$};

\filldraw (1,2) circle (2pt) node[anchor=north] {$5$};
\filldraw (3,2) circle (2pt) node[anchor=north] {$4$};

\filldraw (2,3) circle (2pt) node[anchor=north] {$1$};
\end{tikzpicture}
\end{subfigure}%
\begin{subfigure}[b]{.3\linewidth}
\centering
\begin{tikzpicture}
\draw[thick] (1,0) -- (0,1);
\draw[ultra thick, red] (1,0) -- (2,1);

\draw[thick] (0,1) -- (1,2);
\draw[ultra thick, red] (2,1) -- (1,2);
\draw[thick] (2,1) -- (3,2);

\draw[ultra thick, red] (1,2) -- (2,3);
\draw[thick] (3,2) -- (2,3);

\filldraw (1,0) circle (2pt) node[anchor=north] {$0$};

\filldraw (0,1) circle (2pt) node[anchor=north] {$3$};
\filldraw (2,1) circle (2pt) node[anchor=north] {$2$};

\filldraw (1,2) circle (2pt) node[anchor=north] {$5$};
\filldraw (3,2) circle (2pt) node[anchor=north] {$4$};

\filldraw (2,3) circle (2pt) node[anchor=north] {$1$};
\end{tikzpicture}
\end{subfigure}%
\begin{subfigure}[b]{.3\linewidth}
\centering
\begin{tikzpicture}
\draw[thick] (1,0) -- (0,1);
\draw[ultra thick, red] (1,0) -- (2,1);

\draw[thick] (0,1) -- (1,2);
\draw[thick] (2,1) -- (1,2);
\draw[ultra thick, red] (2,1) -- (3,2);

\draw[thick] (1,2) -- (2,3);
\draw[ultra thick, red] (3,2) -- (2,3);

\filldraw (1,0) circle (2pt) node[anchor=north] {$0$};

\filldraw (0,1) circle (2pt) node[anchor=north] {$3$};
\filldraw (2,1) circle (2pt) node[anchor=north] {$2$};

\filldraw (1,2) circle (2pt) node[anchor=north] {$5$};
\filldraw (3,2) circle (2pt) node[anchor=north] {$4$};

\filldraw (2,3) circle (2pt) node[anchor=north] {$1$};
\end{tikzpicture}
\end{subfigure}
\caption{Chains corresponding to the factorization $(1,2) \in \mathsf Z_P(1)$, where $P$ is the Kunz poset for the numerical semigroup $S = \<6, 9, 20\>$.}
\label{f:rearrangechains}
\end{figure}
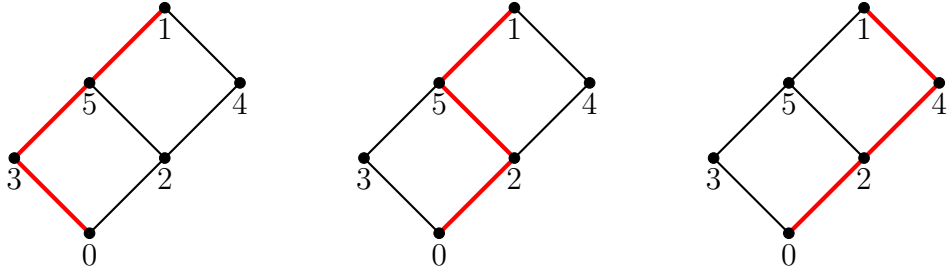

\begin{defn}\label{d:posetminimalpresentation}
Under the kernel congruence $\til = \ker\varphi_P$ of $P$, each relation $z \sim z'$ is called a \emph{trade}.  
A \emph{presentation of $P$} is a set of trades obtained from a generating set for $\til$ by omitting any generators $z \sim z'$ with $\varphi_P(z) = \infty$.  We say a presentation for $P$ is \emph{minimal} if it can be obtained in this way from a generating set for $\til$ that is minimal with respect to containment.  
\end{defn}

\begin{example}\label{e:minpres}
Let $S = \<6,7,8,9\>$, whose Kunz poset is depicted in Figure~\ref{f:findminpres1}(a).  Under Definition~\ref{d:posetfactorizations}, we obtain the factorization sets
\[
\begin{aligned}
\mathsf Z_P(0) &= \{(0,0,0)\}, &
\mathsf Z_P(1) &= \{(1,0,0)\}, &
\mathsf Z_P(2) &= \{(0,1,0)\}, \\
\mathsf Z_P(3) &= \{(0,0,1)\}, &
\mathsf Z_P(4) &= \{(1,0,1), (0,2,0)\}, &
\mathsf Z_P(5) &= \{(0,1,1)\}.
\end{aligned}
\]
and in the corresponding Kunz nilsemigroup $N$, we see $2 \oplus 3 = 5$ since 
\[
(0,1,0) + (0,0,1) = (0,1,1) \in \mathsf Z_P(5),
\]
while $1 \oplus 2 = \infty$ since $(1,0,0) + (0,1,0) = (1,1,0) \in \mathsf Z_P(\infty)$.  
As such, we see the only minimal presentation for $P$ is $\rho = \{((1,0,1),(0,2,0))\}$, since \texttt{Macaulay2}~\cite{M2} can be used (see below) to compute
\[
\ker \varphi_P = \<((1,0,1),(0,2,0)), ((2,0,0),(0,0,2)), ((2,0,0),(1,1,0)), ((2,0,0),(1,1,1))\>,
\]
and only the first relation involves factorizations outside of $\mathsf Z_P(\infty)$.  For instance, one can use the following \texttt{Macaulay2} code, which obtains a minimal generating set for $\varphi_P$ by computing a free resolution over $R = \QQ[x_0, x_1, \ldots, x_k]$ of the ideal 
$$J = I_S + \<x_0x_i - x_0 : i = 0,1,\ldots,k\>$$
and extracting $1^\textnormal{st}$ syzygies, where $I_S$ is the kernel of the map $R \to \QQ[t]$ given by $x_i \mapsto t^{n_i}$ (see~\cite{nsbettisurvey} for background on free resolutions in this context).  A minimal generating set for $I_S$ has one binomial for each trade in a minimal presentation of $S$, and the additional generators of $J$ ensure that any two monomials involving $x_0$ are equal modulo $J$.  

\smallskip
\begin{verbatim}
Q = QQ[x0,x1,x2,x3];  T = QQ[t];
phiS = map(T, Q, matrix {{t^6, t^7, t^8, t^9}});
J = ker(phiS) + ideal(x0^2 - x0, x0*x1 - x0, x0*x2 - x0, x0*x3 - x0);
R = QQ[y1,y2,y3];  U = Q/J;
phiN = map(U, R, matrix {{x1, x2, x3}});
FR = resolution(ker phiN);  FR.dd_1
\end{verbatim}
\smallskip

Next, let $P$ denote the Kunz poset depicted in Figure~\ref{f:findminpres1}(b), and $N$ its Kunz nilsemigroup.  
Only 2 elements of $P$ are not uniquely factorable, and have factorization sets
\[
\mathsf Z_P(6) = \{(2,0,0,0),(0,0,0,2)\}
\qquad \text{and} \qquad
\mathsf Z_P(1) = \{(3,0,0,0),(1,0,0,2),(0,1,1,0)\},
\]
respectively.  
The kernel of $\varphi_P: \NN^4 \to P \cup \{\infty\}$ has 9~generators, but all but 2 contain factorizations of $\mathsf Z_P(\infty)$, occuring at $1 \in P$ and $6 \in P$, respectively.  This yields 
\[
\begin{aligned}
&\{((2,0,0,0),(0,0,0,2)), ((1,0,0,2),(0,1,1,0))\}
\quad \text{and} \\
&\{((2,0,0,0),(0,0,0,2)), ((3,0,0,0),(0,1,1,0))\}
\end{aligned}
\]
as the possible minimal presentations of $P$.  
\end{example}

Due to the inductive nature of the proof of Theorem~\ref{t:allminimalpresentations} given in \cite[Lemma~2.1]{deltabetti}, the cancellativity hypothesis is only used when examining an element's factorization graph and divisors.  As such, one may apply an identical argument to non-nil elements of a semigroup $N$ that is only partly cancellative, resulting in an analogous graph-theoretic characterization of the relations in a minimal presentation of $N$ that occur at non-nil elements.  The following thus holds with an identical proof to that of \cite[Lemma~2.1]{deltabetti}.  

\begin{prop}\label{p:kunzminpresnabla}
Fix a Kunz poset $P$.  A set $\rho$  of relations between factorizations of elements of $P$ is a presentation of $P$ if and only if for every $p \in P$, a connected graph is obtained from $\nabla_p$ by adding an edge for each pair of factorizations of $p$ in~$\rho$.  Furthermore, a presentation $\rho$ of $P$ is minimal if and only if for every $p \in P$, the number of connected components in $\nabla_p$ is one more than the number of relations in $\rho$ containing factorizations of $p$.  
\end{prop}

As is likely not surprising, factorizations of Ap\'ery set elements in a numerical semigroup $S$ coincide with factorizations of elements of the corresponding Kunz poset $P$, as the following theorem indicates.

\begin{thm}\label{t:posettosemigroupminpres}
Fix a numerical semigroup $S$ with Kunz poset $P$.  Writing the Ap\'ery set of $S$ as $\Ap(S;m) = \{0, a_1, \ldots, a_{m-1}\}$ with $a_i \equiv i \bmod m$ for each $i$, we have
$$\mathsf Z_S(a_i) = \{(0, z_1, \ldots, z_k) : z \in \mathsf Z_P(i)\}$$
for each $i$.  Moreover, given any minimal presentation $\rho$ of $S$, the set 
$$\rho' = \{((z_1, \ldots, z_k), (z_1', \ldots, z_k')) : (z, z') \in \rho \text{ with } \varphi_S(z) \in \Ap(S; m)\}$$
is a minimal presentation for $P$.  
\end{thm}

\begin{proof}
Let $N = P \cup \{\infty\}$ denote the Kunz nilsemigroup corresponding to $P$.  The map $f:S \to N$ given by 
$$
n \mapsto \begin{cases}
\ol n & \text{if } n \in \Ap(S;m); \\
\infty & \text{otherwise},
\end{cases}$$
where $\ol n \in \ZZ_m$ denotes the equivalence class of $n$ modulo $m$, is a semigroup homomorphism that restricts to a bijection $\Ap(S; m) \to P$ and has $f^{-1}(\infty) = S \setminus \Ap(S; m)$.  This implies the first claim.  Moreover, $\ker \varphi_S \subseteq \ker(f \circ \varphi_S)$, and since $f$ is injecive on $\Ap(S; m)$, any relations in $\ker(f \circ \varphi_S) \setminus \ker \varphi_S$ occur between factorizations of elements of $S \setminus \Ap(S; m)$.  
As such, fixing a minimal presentation $\rho$ for $S$ and letting $\rho'$ be defined as above, we see two factorizations $z, z' \in \mathsf Z_S(n)$ with $n \in \Ap(S;m)$ are connected by a sequence of trades in $\rho$ if and only if the factorizations $(z_1, \ldots, z_k)$ and $(z_1', \ldots, z_k')$ of $f(n) \in P$ are connected by a sequence of trades in $\rho'$.  
This completes the proof.  
\end{proof}

\begin{example}\label{e:posettosemigroupminpres}
Returning briefly to $S = \<6,7,8,9\>$ as in Examples~\ref{e:groupconefacelattice} and~\ref{e:minpres}, 
a~full minimal presentation of $S$, obtained with~\cite{numericalsgpsgap}, is comprised of the relations
\[
\begin{array}{l@{\qquad}l}
((0, 1, 0, 1), (0, 0, 2, 0)), 
&
((1, 0, 0, 1), (0, 1, 1, 0)), 
\\
((1, 0, 1, 0), (0, 2, 0, 0)), 
&
((3, 0, 0, 0), (0, 0, 0, 2)).
\end{array}
\]
From this, we obtain the minimal presentation of $P$ in Example~\ref{e:minpres} by first eliminating all relations except the first (as they occur at the Betti elements $14, 15, 18 \notin \Ap(S;6)$), and then omitting the preceding $0$ from both remaining factorizations.  
\end{example}

\begin{remark}\label{r:reesequotients}
The map $f$ in the proof of Theorem~\ref{t:posettosemigroupminpres} identifies the Kunz nilsemigroup of a numerical semigroup $S$ as a \emph{Rees quotient}, wherein the elements of a semigroup ideal $I \subset S$ (in this case, $I = S \setminus \Ap(S;m)$) are identified into a single (nil) element.  Rees quotients of numerical semigroups were investigated in~\cite{reesquotientsns}, although the primary focus was on ideals of the form $I = S \cap \ZZ_{\ge t}$ for some $t \in \ZZ_{\ge 1}$, a form which the complement of $\Ap(S;m)$ doesn't fit unless $m = \mathsf m(S)$.  
\end{remark}

\section{A combinatorial formula for face dimension}
\label{sec:facedim}

In this section, we present a combinatorial method of obtaining the dimension of a face of the Kunz polyhedron from its corresponding poset.

\begin{example}\label{e:facedimcomparison}
Consider the Kunz posets $P_1$ and $P_2$, depicted in Figures~\ref{sf:nosemigroupsexample} and~\ref{sf:bettidimension}, respectively.
Although $P_1$ and $P_2$ have identical cardinality and number of atoms, the dimensions of their corresponding faces $F_1$ and $F_2$ in $P_8$ are of different dimension.  In particular, $\dim F_1 = 2$, while $\dim F_2 = 4$.  As we will see in Theorem~\ref{t:posetdimension}, this is closely related to the fact that $P_1$ has 2 relations in its minimal presentation (as seen in Example~\ref{e:minpres}), while $P_2$ has an empty minimal presentation.  
%
%
%
%
\end{example}

\begin{defn}\label{d:hmatrix}
Fix a face $F \subset P_m$ with corresponding Kunz poset $P$ on $\ZZ_m$.  Suppose that $P$ has $k$ atoms and that $F$ is contained in $i$ facets.  
\begin{enumerate}[(a)]
\item 
The \emph{hyperplane matrix} of $F$ is the matrix $H_F \in \ZZ^{i\times (m-1)}$ whose columns are indexed by the nonzero elements of $\ZZ_m$ and whose rows are given by the equations of the facets containing~$F$.  

\item
Given any finite presentation $\rho$ of $P$, the matrix $M_\rho \in \ZZ^{|\rho| \times k}$ whose columns are indexed by the atoms of $P$ and whose rows have the form $z - z'$ for $(z, z') \in \rho$ is called a \emph{presentation matrix} of $P$.  

\end{enumerate}
\end{defn}

\begin{thm}\label{t:posetdimension}
If $F \subset P_m$ is a face with corresponding Kunz poset $P$ on $\ZZ_m$, then
\[
\dim F = k - \rk(M_\rho)
\]
where $k$ is the number of atoms of $P$ and $\rho$ is any finite presentation of $P$.  
\end{thm}

Before proving Theorem~\ref{t:posetdimension}, we provide an example to illustrate the proof structure.

\begin{example}\label{e:posetdimension}
Consider the Kunz poset $P$ depicted in Figure~\ref{sf:nosemigroupsexample}, and let 
\[
H_F =
\begin{pmatrix}
\phantom{-}0 & \phantom{-}0 & \phantom{-}2 & \phantom{-}0 & \phantom{-}0 & -1 & \phantom{-}0\\
\phantom{-}0 & -1 & \phantom{-}1 & \phantom{-}0 & \phantom{-}0 & \phantom{-}0 & \phantom{-}1\\
-1 & \phantom{-}0 & \phantom{-}0 & \phantom{-}1 & \phantom{-}1 & \phantom{-}0 & \phantom{-}0\\
\phantom{-}0 & \phantom{-}0 & \phantom{-}0 & \phantom{-}0 & \phantom{-}0 & -1 & \phantom{-}2\\
-1 & \phantom{-}0 & \phantom{-}1 & \phantom{-}0 & \phantom{-}0 & \phantom{-}1 & \phantom{-}0\\
-1 & \phantom{-}1 & \phantom{-}0 & \phantom{-}0 & \phantom{-}0 & \phantom{-}0 & \phantom{-}1\\
\end{pmatrix}
\]
denote the hyperplane matrix of the corresponding face $F \subseteq P_8$, with one row for each facet of $P_8$ containing $F$.  The second row, for instance, indicates all points in $F$ satisfy $x_3 + x_7 = x_2$, and that $3,7 \preceq 2$ in $P$.  
We begin by fixing a linear extension of $P$ where the atoms ($3$, $4$, $5$, and $7$) occur before the non-zero non-atoms ($2$, $6$, and $1$), e.g., 
\[
0 \preceq 3 \preceq 4 \preceq 5 \preceq 7 \preceq 2 \preceq 6 \preceq 1.  
\]
Rearranging the columns of $H_F$ to list the non-atoms in the above order followed by the atoms, 
then subsequently rearranging the rows so that each of the first 3 columns has $-1$ on the diagonal, yields
\[
\begin{blockarray}{ccc|cccc}
\phantom{-}2 & \phantom{-}6 & \phantom{-}1 & \phantom{-}3 & \phantom{-}4 & \phantom{-}5 & \phantom{-}7\\
\begin{block}{(ccc|cccc)}
-1 & \phantom{-}0 & \phantom{-}0 & \phantom{-}1 & \phantom{-}0 & \phantom{-}0 & \phantom{-}1\\
\phantom{-}0 & -1 & \phantom{-}0 & \phantom{-}2 & \phantom{-}0 & \phantom{-}0 & \phantom{-}0\\
\phantom{-}1 & \phantom{-}0 & -1 & \phantom{-}0 & \phantom{-}0 & \phantom{-}0 & \phantom{-}1\\
\BAhline
\phantom{-}0 & -1 & \phantom{-}0 & \phantom{-}0 & \phantom{-}0 & \phantom{-}0 & \phantom{-}2\\
\phantom{-}0 & \phantom{-}1 & -1 & \phantom{-}1 & \phantom{-}0 & \phantom{-}0 & \phantom{-}0\\
\phantom{-}0 & \phantom{-}0 & -1 & \phantom{-}0 & \phantom{-}1 & \phantom{-}1 & \phantom{-}0\\
\end{block}
\end{blockarray}.
\]
The upper-left submatrix is assuredly lower diagonal by the choice in column order.  As such, we may clear any positive entries from the first 3 columns by adding positive integer multiples of the first 3 rows accordingly.  In the resulting matrix, each row has precisely one nonzero entry in the first 3 columns (specifically, $-1$), and the remaining row entries comprise a factorization for the label of the column with $-1$.  
\[
\begin{blockarray}{ccc|cccc}
\phantom{-}2 & \phantom{-}6 & \phantom{-}1 & \phantom{-}3 & \phantom{-}4 & \phantom{-}5 & \phantom{-}7\\
\begin{block}{(ccc|cccc)}
-1 & \phantom{-}0 & \phantom{-}0 & \phantom{-}1 & \phantom{-}0 & \phantom{-}0 & \phantom{-}1\\
\phantom{-}0 & -1 & \phantom{-}0 & \phantom{-}2 & \phantom{-}0 & \phantom{-}0 & \phantom{-}0\\
\phantom{-}0 & \phantom{-}0 & -1 & \phantom{-}1 & \phantom{-}0 & \phantom{-}0 & \phantom{-}2\\
\BAhline
\phantom{-}0 & -1 & \phantom{-}0 & \phantom{-}0 & \phantom{-}0 & \phantom{-}0 & \phantom{-}2\\
\phantom{-}0 & \phantom{-}0 & -1 & \phantom{-}3 & \phantom{-}0 & \phantom{-}0 & \phantom{-}0\\
\phantom{-}0 & \phantom{-}0 & -1 & \phantom{-}0 & \phantom{-}1 & \phantom{-}1 & \phantom{-}0\\
\end{block}
\end{blockarray}.
\]
At this point, subtracting one of the first 3 rows from each of the remaining rows yields
\[
\begin{blockarray}{ccc|cccc}
\phantom{-}2 & \phantom{-}6 & \phantom{-}1 & \phantom{-}3 & \phantom{-}4 & \phantom{-}5 & \phantom{-}7\\
\begin{block}{(ccc|cccc)}
-1 & \phantom{-}0 & \phantom{-}0 & \phantom{-}1 & \phantom{-}0 & \phantom{-}0 & \phantom{-}1\\
\phantom{-}0 & -1 & \phantom{-}0 & \phantom{-}2 & \phantom{-}0 & \phantom{-}0 & \phantom{-}0\\
\phantom{-}0 & \phantom{-}0 & -1 & \phantom{-}1 & \phantom{-}0 & \phantom{-}0 & \phantom{-}2\\
\BAhline
\phantom{-}0 & \phantom{-}0 & \phantom{-}0 & -2 & \phantom{-}0 & \phantom{-}0 & \phantom{-}2\\
\phantom{-}0 & \phantom{-}0 & \phantom{-}0 & \phantom{-}2 & \phantom{-}0 & \phantom{-}0 & -2\\
\phantom{-}0 & \phantom{-}0 & \phantom{-}0 & -1 & \phantom{-}1 & \phantom{-}1 & -2\\
\end{block}
\end{blockarray}
\]
wherein the upper-left block is $-I_3$, the lower-left block is the zero-matrix, each row in the upper-right block is a factorization of a non-atom of $P$, and the lower-right block is a presentation matrix $M_\rho$ for $P$ obtained for the (non-minimal) presentation
\[
\rho = \big\{
((0,0,0,2), (2,0,0,0)), 
\quad
((2,0,0,0), (0,0,0,2)),
\quad
((0,1,1,0), (1,0,0,2))
\big\}.
\]
This is the core of the proof of Theorem~\ref{t:posetdimension}:\ the row operations used to transform the upper-left block into $-I_3$ effectively perform the same recursive process used to obtain poset element factorizations described in Remark~\ref{r:chains}.  
\end{example}

\begin{proof}[Proof of Theorem~\ref{t:posetdimension}]
Fix a linear extension $0 \preceq p_1 \preceq \cdots \preceq p_{m-1}$ of $P$ where $p_1, \ldots, p_k$ are the atoms of $P$ (for instance, one could order by maximum factorization length and then break ties arbitrarily, since no elements with equal maximum factorization length are comparable).  
In what follows, we will perform elementary row and column operations on $H_F$ to obtain a matrix of the form
\begin{equation}\label{eq:matrixform}
\begin{blockarray}{ccc|ccc}
p_{k+1} & \cdots & p_{m-1} & p_1 & \cdots & p_k\\
\begin{block}{(ccc|ccc)}
\BAmulticolumn{3}{c|}{\multirow{3}{*}{$\mathbigger{-I_{m-1-k}}$}} 
&\BAmulticolumn{3}{c}{\multirow{3}{*}{$\mathbigger{A}$}}\\
&&&&\\
&&&&\\
\BAhline
\BAmulticolumn{3}{c|}{\multirow{3}{*}{$\mathbigger{0}$}} 
& \BAmulticolumn{3}{c}{\multirow{3}{*}{$\mathbigger{M_\rho}$}}\\
&&&&\\
&&&&\\
\end{block}
\end{blockarray}
\end{equation}
where $M_\rho$ is some presentation matrix of $P$ and $A$ is a matrix whose $i$th row is a factorization of $p_{k+i}$ for each $i \le m - k - 1$.  

First, reorder the columns of $H_F$ as indicated in~\eqref{eq:matrixform}.  
By construction,  the left hand columns are precisely those with a $-1$ in at least one entry.
As such, we can reorder the rows so that the upper-left $(m-1-k) \times (m-1-k)$ submatrix has $-1$'s along the diagonal and (necessarily, due to the chosen column order) is lower triangular with non-negative integers below the diagonal.  Now, we eliminate each positive entry below the diagonal in this square submatrix by adding the appropriate row above, yielding the desired upper blocks of the matrix in~\eqref{eq:matrixform}.

Now, each row in the lower blocks of the matrix has $-1$ in exactly one entry in the lower left block.  We can thus eliminate all remaining positive entries in the left hand blocks by adding one or more of the first $m - k - 1$ rows as needed, leaving exactly one nonzero entry in each row (namely, $-1$ in some column $p_w$), and non-negative integers $w_1' + v_1, \ldots, w_k' + v_k$ in the last $k$ columns (as depicted below) with the property that $p_w = (w_1' + v_1)p_1 + \cdots + (w_k' + v_k)p_k$ in $P$ (that is, $(w_1' + v_1, \ldots, w_k' + v_k) \in \mathsf Z_P(p_w)$).  
\[
\begin{blockarray}{ccccc|ccc}
& p_{v} & & p_{w} & & p_1 & \cdots & p_k\\
\begin{block}{(ccccc|ccc)}
 & -1 &  & & & v_1 & \cdots & v_k\\
& &  & -1 &  & w_1 & \cdots & w_k\\
\BAhline
\qquad & 1 & \qquad & -1 & \qquad & w_1' & \cdots & w_k'\\
\end{block}
\end{blockarray}
\qquad
\rightsquigarrow
\qquad
\begin{blockarray}{ccccc|ccc}
& p_{v} & & p_{w} & & p_1 & \cdots & p_k\\
\begin{block}{(ccccc|ccc)}
 & -1 &  & & & v_1 & \cdots & v_k\\
& &  & -1 &  & w_1 & \cdots & w_k\\
\BAhline
\qquad & 0 & \qquad & -1 & \qquad & w_1' + v_1 & \cdots & w_k' + v_k\\
\end{block}
\end{blockarray}
\]

At this stage of the row reduction process, each row has exactly one nonzero entry in the left half of the matrix (namely, a $-1$ in some column $p_w$), and a factorization of $p_w$ in the last $k$ columns.  
We claim that, for each non-atom $p_w$ and each connected component $Z$ of $\nabla_{p_w}$, there exists a row with $-1$ in the column $p_w$ whose factorization in the last $k$ columns lies in $Z$.  Indeed, suppose $i \in \supp(z)$ for some $z \in Z$.  Then $p_i \preceq p_w$ in $P$, so one of the rows of the original matrix $H_F$ corresponds to the equality $x_{p_i} + x_{p_w - p_i} = x_{p_w}$ satisfied by all points in $F$.  In the current matrix, that row still has $-1$ in the column $p_w$, and the factorization $z'$ of $p_w$ in the last $k$ columns has a positive value in the column $p_i$.  As such, $z' \in Z$, and the claim is proven.  

For the last stage of the row-reduction process, from each row in the lower half of the matrix with $-1$ in the column $p_w$, subtract the appropriate row from the top half of the matrix, yielding 0 in the column $p_w$ and a trade for $P$ in the final $k$ columns, as depicted below.  
Let $Z$ denote the connected component of $\nabla_{p_w}$ containing the factorization of $p_w$ in the rows of $A$.  By the claim in the previous paragraph, for each connected component $Z'$ of $\nabla_{p_w}$ other than $Z$, there is at least one trade between a factorization in $Z$ and a factorization in $Z'$ occuring in the lower right block.  As such, the lower right block $M_\rho$ is in fact a presentation matrix of $P$ for some finite presentation $\rho$ of $P$.  

\[
\begin{blockarray}{ccc|ccc}
 & p_{w} & & p_1 & \cdots & p_k\\
\begin{block}{(ccc|ccc)}
  & -1 &  & w_1 & \cdots & w_k\\
\BAhline
 \qquad & -1 & \qquad & w_1' & \cdots & w_k'\\
\end{block}
\end{blockarray}
\qquad
\rightsquigarrow
\qquad
\begin{blockarray}{ccc|ccc}
& p_{w} & & p_1 & \cdots & p_k\\
\begin{block}{(ccc|ccc)}
& -1 &  & w_1 & \cdots & w_k\\
\BAhline
\qquad & 0 & \qquad & w_1' - w_1 & \cdots & w_k' - w_k\\
\end{block}
\end{blockarray}
\]
The above steps ensure the matrix now has the form in~\eqref{eq:matrixform}.  
Substitution then yields
\begin{align*}
\dim F &= m-1 - \rk(H_F) \\
&= m - 1 - \big(m - 1 - k + \rk(M_\rho)\big) \\
&= k - \rk(M_\rho).
\end{align*}

We complete the proof by noting that any finite presentation $\rho'$ of $P$ has as a subset a minimal presentation $\rho$ of $P$, and any row of $M_{\rho'}$ not appearing in $M_\rho$ is an integral linear combination of rows of $M_\rho$, since the corresponding trade in $\rho'$ can be obtained from trades in $\rho$ via translation and transitive closure.  As such, $M_\rho$ and $M_{\rho'}$ have identical row span, and thus $\rk(M_{\rho'}) = \rk(M_\rho)$.  On the other hand, if $\rho$ is a minimal presentation of $P$ and $z,z' \in \mathsf Z_P(p)$ for some $p \in P$, then $(z,z')$ can be obtained via translation and transitivity from trades in $\rho$, so $z - z'$ is an integral linear combination of the rows of $M_\rho$ and visa versa.  As such, if $\rho$ and $\rho'$ are both minimal presentations of $P$, then any row of $M_{\rho'}$ is an integral linear combination of rows of $M_\rho$, so $\rk(M_{\rho'}) = \rk(M_\rho)$.  This ensures any two presentation matrices have equal rank, thereby completing the proof.  
\end{proof}

\begin{remark}
Theorem~\ref{t:posetdimension} yields a more practical method of computing (by hand) the dimension of the face $F \subseteq P_m$ corresponding to a given Kunz poset $P$.  The first step involves finding the minimal presentation of $P$, which can be readily obtained from the factorization graphs of the non-nil elements of $P$ (of which there are only finitely many), and the resulting matrix $M_\rho$ is substantially smaller than the full hyperplane matrix $H_F$ of $F$.  
\end{remark}

\section{Counting minimal relations for semigroups in a given face}
\label{sec:minpressize}

Any two numerical semigroups $S$ and $S'$ lying in the same face $F$ of the Kunz polyhedron share several invariant values (for instance, $\mathsf e(S) = \mathsf e(S')$ and $\mathsf t(S) = \mathsf t(S')$).  
In this section, we prove $S$ and $S'$ also have minimal presentations of equal cardinality.  

\begin{notation}\label{n:fhat}
Throughout this section, unless otherwise stated, $N$ denotes a (not necessarily Kunz) finite partly cancellative nilsemigroup with $\mathcal A(N) = \{p_1, \ldots, p_k\}$.  
Given $z = (z_0, z_1, \ldots, z_k) \in \NN^{k+1}$ or $z = (z_1, \ldots, z_k) \in \NN^k$, let 
\[
\wh z = (z_1, \ldots, z_k) \in \NN^k.
\]
The \emph{support} of $z$ is the set 
\[
\supp(z) = \{i : z_i > 0\} \subset \{0, \ldots, k\}.
\]
Given a set $Z \subset \NN^{k+1}$, the \emph{support} of $Z$ is the set 
\[
\supp(Z) = \bigcup_{z \in Z} \supp(z),
\]
and for each $i \in \supp(Z)$, define 
\[
Z - e_i = \{z - e_i : z \in Z \text{ with } i \in \supp(z)\}.
\]
Lastly, if $N$ is a Kunz nilsemigroup with $G = N \setminus \{\infty\}$, then define
\[
\ol z = z_1p_1 + \cdots + z_kp_k \in G.
\]
In particular, if $z \in \mathsf Z_N(p)$ with $p \ne \infty$, then $\ol z = p$.  
\end{notation}

\begin{defn}\label{d:outerbetti}
An \emph{outer Betti element} of $N$ is a set $B \subset \mathsf Z_N(\infty)$ such that 
\begin{enumerate}[(i)]
\item 
for each $i \in \supp(B)$, we have $B - e_i = \mathsf Z_N(p)$ for some $p \in N \setminus \{\infty\}$, and
\item 
the graph $\nabla_B$ is connected. 
\end{enumerate}
Likewise, an \emph{outer Betti element} of a Kunz poset $P$ is an outer Betti element of its corresponding Kunz nilsemigroup.  
\end{defn}


We now explore several examples illustrating the nuances of Definition~\ref{d:outerbetti} and the intuition behind its use in obtaining minimal presentations of numerical semigroups.  

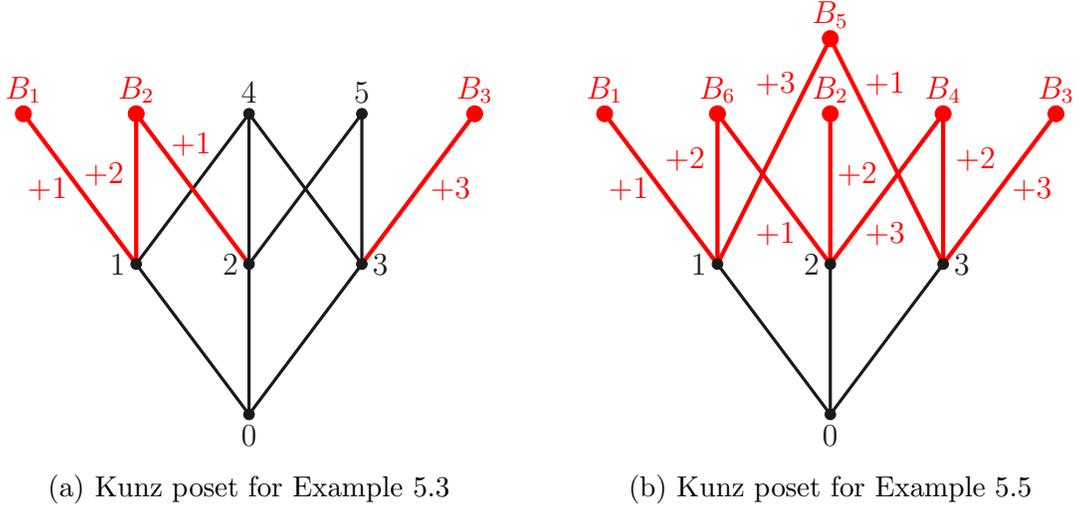
\begin{figure}[t]
\centering
\begin{subfigure}[b]{.45\linewidth}
\centering
\begin{tikzpicture}
\def\outerBCol{red}
\def\posetCol{black!90}
\coordinate (0) at (1,0);
\coordinate (1) at (-0.5,2);
\coordinate (2) at (1,2);
\coordinate (3) at (2.5,2);
\coordinate (4) at (1,4);
\coordinate (5) at (2.5,4);
\coordinate (B1) at (-2, 4);
\coordinate (B2) at (-0.5, 4);
\coordinate (B3) at (4, 4);

\draw[very thick,\posetCol] (0) -- (1); 
\draw[very thick,\posetCol] (0) -- (2); 
\draw[very thick,\posetCol] (0) -- (3); 

\draw[very thick,\posetCol] (1) -- (4); 
\draw[very thick,\posetCol] (2) -- (4); 
\draw[very thick,\posetCol] (2) -- (5); 
\draw[very thick,\posetCol] (3) -- (4); 
\draw[very thick,\posetCol] (3) -- (5); 

\draw[ultra thick, \outerBCol] (1) -- (B1) node[midway,left] {$+1$};
\draw[ultra thick, \outerBCol] (1) -- (B2) node[pos=.6,left] {$+2$};
\draw[ultra thick, \outerBCol] (2) -- (B2) node[pos=.8,right] {$+1$};
\draw[ultra thick, \outerBCol] (3) -- (B3) node[midway,right] {$+3$};

\filldraw[\posetCol] (0) circle (2pt) node[anchor=north] {$0$};

\filldraw[\posetCol] (1) circle (2pt) node[anchor=east] {$1$};
\filldraw[\posetCol] (2) circle (2pt) node[anchor=east] {$2$};
\filldraw[\posetCol] (3) circle (2pt) node[anchor=west] {$3$};

\filldraw[\posetCol] (4) circle (2pt) node[anchor=south] {$4$};
\filldraw[\posetCol] (5) circle (2pt) node[anchor=south] {$5$};

\filldraw[\outerBCol] (B1) circle (3pt) node[anchor=south] {$B_1$};
\filldraw[\outerBCol] (B2) circle (3pt) node[anchor=south] {$B_2$};
\filldraw[\outerBCol] (B3) circle (3pt) node[anchor=south] {$B_3$};
\end{tikzpicture}
\caption{Kunz poset for Example~\ref{e:outerbettisimple}}
\label{sf:outerbettisimple}
\end{subfigure}%
\hspace{0.04\linewidth}
\begin{subfigure}[b]{.45\linewidth}
\centering
\begin{tikzpicture}

\def\outerBCol{red}
\def\posetCol{black!90}

\coordinate (0) at (3,0);
\coordinate (1) at (1.5,2);
\coordinate (2) at (3,2);
\coordinate (3) at (4.5,2);
\coordinate (B1) at (0, 4);
\coordinate (B2) at (3, 4);
\coordinate (B3) at (6, 4);
\coordinate (B4) at (4.5, 4);
\coordinate (B5) at (3, 5);
\coordinate (B6) at (1.5, 4);

\draw[very thick,\posetCol] (0) -- (1);
\draw[very thick,\posetCol] (0) -- (2);
\draw[very thick,\posetCol] (0) -- (3);

\draw[ultra thick,\outerBCol] (1) -- (B1) node[midway,left] {$+1$};
\draw[ultra thick,\outerBCol] (1) -- (B6) node[pos=0.7,left] {$+2$};
\draw[ultra thick,\outerBCol] (1) -- (B5) node[pos=0.8,left] {$+3$};
\draw[ultra thick,\outerBCol] (2) -- (B2) node[pos=0.6,right] {$\!+2$};
\draw[ultra thick,\outerBCol] (2) -- (B6) node[pos=0.2,left] {$+1$};
\draw[ultra thick,\outerBCol] (2) -- (B4) node[pos=0.2,right] {$+3$};
\draw[ultra thick,\outerBCol] (3) -- (B3) node[midway,right] {$+3$};
\draw[ultra thick,\outerBCol] (3) -- (B4) node[pos=0.7,right] {$+2$};
\draw[ultra thick,\outerBCol] (3) -- (B5) node[pos=0.8,right] {$+1$};

\filldraw[\posetCol] (0) circle (2pt) node[anchor=north] {$0$};

\filldraw[\posetCol] (1) circle (2pt) node[anchor=east] {$1$};
\filldraw[\posetCol] (2) circle (2pt) node[anchor=east] {$2$};
\filldraw[\posetCol] (3) circle (2pt) node[anchor=west] {$3$};

\filldraw[\outerBCol] (B1) circle (3pt) node[anchor=south] {$B_1$};
\filldraw[\outerBCol] (B2) circle (3pt) node[anchor=south] {$B_2$};
\filldraw[\outerBCol] (B3) circle (3pt) node[anchor=south] {$B_3$};
\filldraw[\outerBCol] (B4) circle (3pt) node[anchor=south] {$B_4$};
\filldraw[\outerBCol] (B5) circle (3pt) node[anchor=south] {$B_5$};
\filldraw[\outerBCol] (B6) circle (3pt) node[anchor=south] {$B_6$};
\end{tikzpicture}
\caption{Kunz poset for Example~\ref{e:outerbettimed}}
\label{sf:outerbettimed}
\end{subfigure}%
\caption{Kunz posets with outer Betti elements and relations depicted by thick (red) lines.}
\label{f:outerbettisshort}
\end{figure}

\begin{example}\label{e:outerbettisimple}
Consider the semigroup $S = \<6,7,8,9\>$ from Example~\ref{e:minpres}, and let $P$ and $N$ denote the corresponding Kunz poset and nilsemigroup, respectively.  Let us first find the outer Betti elements of $N$.  By Definition~\ref{d:outerbetti}(i), each factorization appearing in an outer Betti element must lie in $\mathsf Z_N(\infty)$, but with the property that decrementing any nonzero entry yields a factorization outside of $\mathsf Z_N(\infty)$.  The factorizations with this property are
$$
(2,0,0), 
\qquad
(1,1,0), 
\qquad
(0,2,1), 
\qquad \text{and} \qquad
(0,0,2).
$$
By Definition~\ref{d:outerbetti}(ii), any two of these that lie in the same outer Betti element must have overlapping support.  Furthermore, $(2,0,0)$ and $(1,1,0)$ also cannot appear together in an outer Betti element $B$ by Definition~\ref{d:outerbetti}(i) since $B - e_1 = \{(1,0,0), (0,1,0)\}$ contains factorizations of distinct elements of $N$.  We also note the elements of $N$ all have singleton length set, so by Definition~\ref{d:outerbetti}(i) any non-singleton outer Betti element must consist of factorizations of equal length, and since $(0,2,1)$ is the only factorization above with length 3, it cannot occur in a non-singleton outer Betti element.  
As~such, all outer Betti elements of $N$ are singleton.  One can then check
$$
B_1 = \{(2, 0, 0)\},
\qquad
B_2 = \{(1, 1, 0)\},
\qquad \text{and} \qquad
B_3 = \{(0, 0, 2)\}$$
all satisfy Definition~\ref{d:outerbetti}, while $B = \{(0,2,1)\}$ does not, since 
$$B - e_3 \subsetneq \mathsf Z_N(4) = \{(0,2,0), (1,0,1)\}.$$

We now return our attention to $S$.  The unique minimal presentation $\rho$ of $S$ has the relation $((0, 0, 2, 0), (0, 1, 0, 1))$ occurring at the Betti element $16 \in \Ap(S;6)$, along~with 
\[
((1, 0, 1, 0), (0, 2, 0, 0)), \quad 
((1, 0, 0, 1), (0, 1, 1, 0)), \quad
\text{and} \quad
((3, 0, 0, 0), (0, 0, 0, 2))
\]
occurring at the Betti elements $14, 15, 18 \notin \Ap(S;6)$, respectively.  
Comparing with the outer Betti elements of $N$, we see the singleton factorization in $B_1$, $B_2$, and $B_3$ each coincide with $\wh z\,'$ where $z'$ is the multiplicity-free factorization of 14, 15, and 18, respectively.  Here, the uniqueness of $\rho$ can be seen, in part, as a consequence of the fact that every outer Betti element of $N$ is singleton.  The poset $P$, along with its outer Betti elements, is depicted in Figure~\ref{sf:outerbettisimple}.  
\end{example}

\begin{example}\label{e:outerbettimerge}
Consider the numerical semigroup $S = \<9, 20, 30, 35\>$.  Its Kunz poset~$P$ (depicted in Figure~\ref{sf:outerbettimerge}) and Kunz nilsemigroup $N$ have outer Betti elements
\[
\begin{array}{r@{}c@{}l@{\qquad}r@{}c@{}l}
B_1 &{}={}& \{(0,1,1)\},
&
B_3 &{}={}& \{(1,2,0), (4,0,0)\},
\\
B_2 &{}={}& \{(2,0,1)\},
&
B_4 &{}={}& \{(3,1,0), (1,0,2), (0,3,0)\}.
\end{array}
\]
There are a total of 6 minimal presentations of $S$, each with 6 relations.  All contain 
\[
(0,3,0,0) \sim (0,0,2,0)
\qquad \text{and} \qquad
(0,2,1,0) \sim (0,0,0,2),
\]
which occur at Ap\'ery set elements.  The other 4 relations are obtained by choosing one factorization from each of $B_1$, $B_2$, $B_3$, and $B_4$ (which occur at the Betti elements $b_1 = 65$, $b_2 = 75$, $b_3 = 80$, and $b_4 = 90$ of $S$, respectively), prepending a 0 to each, then choosing one factorization from each of the sets 
$\mathsf Z_P(2)$, $\mathsf Z_P(3)$, $\mathsf Z_P(8)$, and $\mathsf Z_P(0)$, 
respectively, to pair with, and lastly prepending an appropriate number of copies of~$9$.  For~instance, choosing the first factorization from each $B_i$ above yields the relations
\[
\begin{array}{r@{\qquad}r}
(5,1,0,0) \sim (0,0,1,1),
&
(5,0,1,0) \sim (0,2,0,1),
\\
(5,0,0,1) \sim (0,1,2,0),
&
(10,0,0,0) \sim (0,3,1,0).
\end{array}
\]
In general, the first factorization in each relation above can be (i) any factorization $z$ of $b_i$ with $z_0 > 0$, or (ii) a factorization $z$ where $\hat z$ lies in an outer Betti element distinct from $B_i$ and whose corresponding element of $S$ precedes that of $B_i$ (as in the next example).  For this particular numerical semigroup, no factorizations of the latter form exist, so the above process yields a complete list of minimal presentations of $S$.  
Definition~\ref{d:m-centric} identifies the minimal presentation constructed in this manner.  
\end{example}

\begin{figure}[t]
\centering
\begin{subfigure}[b]{.45\linewidth}
\centering
\begin{tikzpicture}

\def\outerBCol{red}
\def\posetCol{black!90}

\coordinate (0) at (0,0);
\coordinate (2) at (0,2);
\coordinate (3) at (2,2);
\coordinate (8) at (4,2);
\coordinate (4) at (0,4);
\coordinate (5) at (2,4);
\coordinate (1) at (4,4);
\coordinate (6) at (0,6);
\coordinate (7) at (2,6);
\coordinate (B1) at (5, 4);
\coordinate (B2) at (4, 6);
\coordinate (B3) at (0, 8);
\coordinate (B4) at (2, 8);

\draw[very thick,\posetCol] (0) -- (2);
\draw[very thick,\posetCol] (0) -- (3);
\draw[very thick,\posetCol] (0) -- (8);

\draw[very thick,\posetCol] (2) -- (4);
\draw[very thick,\posetCol] (2) -- (1);
\draw[very thick,\posetCol] (2) -- (5);
\draw[very thick,\posetCol] (3) -- (6);
\draw[very thick,\posetCol] (3) -- (5);
\draw[very thick,\posetCol] (8) -- (1); 
\draw[very thick,\posetCol] (8) -- (7); 

\draw[very thick,\posetCol] (4) -- (6); 
\draw[very thick,\posetCol] (4) -- (7); 
\draw[very thick,\posetCol] (5) -- (7); 

\draw[ultra thick, \outerBCol] (8) -- (B1) node[midway,right] {$+3$};
\draw[ultra thick, \outerBCol] (3) -- (B1) node[pos=0.2,right] {$+8$};
\draw[ultra thick, \outerBCol] (4) -- (B2) node[pos=.61,right] {$+8$};
\draw[ultra thick, \outerBCol] (1) -- (B2) node[pos=.5,right] {$+2$};
\draw[ultra thick, \outerBCol] (6) -- (B3) node[pos=.6,left] {$+2$};
\draw[ultra thick, \outerBCol] (5) -- (B3) node[pos=.45,left] {$+3$};
\draw[ultra thick, \outerBCol] (7) -- (B4) node[pos=0.4,left] {$+2$};
\draw[ultra thick, \outerBCol] (6) -- (B4) node[pos=0.8,left] {$+3$};
\draw[ultra thick, \outerBCol] (1) -- (B4) node[pos=0.8,right] {$+8$};

\filldraw[\posetCol] (0) circle (2pt) node[anchor=north] {$0$};

\filldraw[\posetCol] (2) circle (2pt) node[anchor=east] {$2$};
\filldraw[\posetCol] (3) circle (2pt) node[anchor=east] {$3$};
\filldraw[\posetCol] (8) circle (2pt) node[anchor=west] {$8$};

\filldraw[\posetCol] (4) circle (2pt) node[anchor=south east] {$4$};
\filldraw[\posetCol] (5) circle (2pt) node[anchor=east] {$5$};

\filldraw[\posetCol] (1) circle (2pt) node[anchor=west] {$1$};
\filldraw[\posetCol] (6) circle (2pt) node[anchor=south east] {$6$};
\filldraw[\posetCol] (7) circle (2pt) node[anchor=south east] {$7$};

\filldraw[\outerBCol] (B1) circle (3pt) node[anchor=south] {$B_1$};
\filldraw[\outerBCol] (B2) circle (3pt) node[anchor=south] {$B_2$};
\filldraw[\outerBCol] (B3) circle (3pt) node[anchor=south] {$B_3$};
\filldraw[\outerBCol] (B4) circle (3pt) node[anchor=south] {$B_4$};
\end{tikzpicture}
\caption{The poset for Example~\ref{e:outerbettimerge}}
\label{sf:outerbettimerge}
\end{subfigure}%
\hspace{0.04\linewidth}
\begin{subfigure}[b]{.45\linewidth}
\centering
\begin{tikzpicture}

\def\outerBCol{red}
\def\posetCol{black!90}

\def\covfive{\posetCol}
\def\covsix{\posetCol}
\def\covseven{\posetCol}
\def\coveight{\posetCol}
\def\covnine{\posetCol}

\coordinate (0) at (0,0);
\coordinate (5) at (0,3);
\coordinate (6) at (1,3);
\coordinate (7) at (2,3);
\coordinate (8) at (3,3);
\coordinate (9) at (4,3);
\coordinate (10) at (0,5);
\coordinate (1) at (2,6);
\coordinate (2) at (3,6);
\coordinate (3) at (4,6);

\coordinate (4) at (0,7);

\draw[very thick,\covfive] (0) -- (5);
\draw[very thick,\covsix] (0) -- (6);
\draw[very thick,\covseven] (0) -- (7);
\draw[very thick,\coveight] (0) -- (8);
\draw[very thick,\covnine] (0) -- (9);

\draw[very thick,\covseven] (5) -- (1);
\draw[very thick,\coveight] (5) -- (2);
\draw[very thick,\covnine] (5) -- (3);
\draw[very thick,\covfive] (5) -- (10);
\draw[very thick,\covsix] (6) -- (1);
\draw[very thick,\covseven] (6) -- (2);
\draw[very thick,\coveight] (6) -- (3);
\draw[very thick,\covnine] (6) -- (4);
\draw[very thick,\covfive] (7) -- (1);
\draw[very thick,\covsix] (7) -- (2);
\draw[very thick,\covseven] (7) -- (3);
\draw[very thick,\coveight] (7) -- (4);
\draw[very thick,\covfive] (8) -- (2);
\draw[very thick,\covsix] (8) -- (3);
\draw[very thick,\covseven] (8) -- (4);
\draw[very thick,\covfive] (9) -- (3);
\draw[very thick,\covsix] (9) -- (4);

\draw[very thick,\covfive] (10) -- (4);

\filldraw[\posetCol] (0) circle (2pt) node[anchor=north] {$0$};

\filldraw[\posetCol] (5) circle (2pt) node[anchor=east] {$5$};
\filldraw[\posetCol] (6) circle (2pt) node[anchor=west] {$6$};
\filldraw[\posetCol] (7) circle (2pt) node[anchor=west] {$7$};
\filldraw[\posetCol] (8) circle (2pt) node[anchor=west] {$8$};
\filldraw[\posetCol] (9) circle (2pt) node[anchor=west] {$9$};

\filldraw[\posetCol] (10) circle (2pt) node[anchor=east] {$10$};

\filldraw[\posetCol] (1) circle (2pt) node[anchor=south] {$1$};
\filldraw[\posetCol] (2) circle (2pt) node[anchor=south] {$2$};
\filldraw[\posetCol] (3) circle (2pt) node[anchor=south] {$3$};
\filldraw[\posetCol] (4) circle (2pt) node[anchor=south] {$4$};

\end{tikzpicture}
\caption{The poset for Example~\ref{e:outerbettipartition}}
\label{sf:outerbettipartition}
\end{subfigure}%
\caption{Kunz posets with outer Betti elements and relations depicted by thick (red) lines.}
\label{f:outerbettistall}
\end{figure}
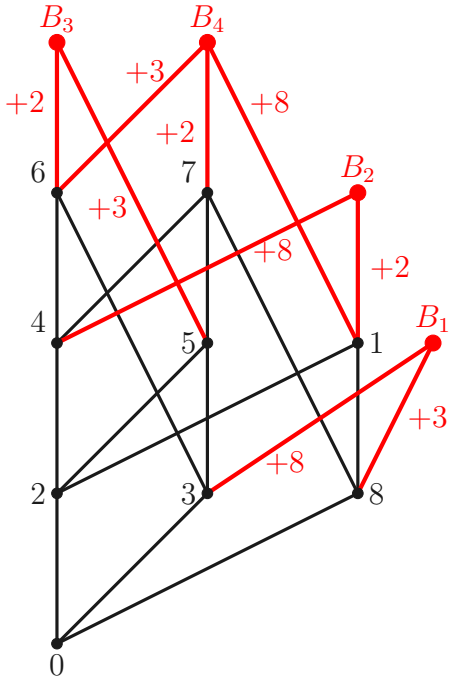
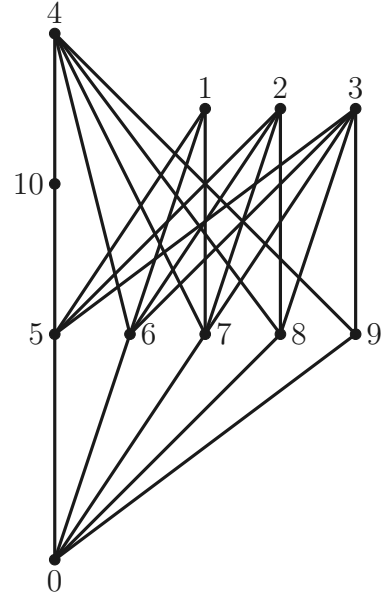

\begin{example}\label{e:outerbettimed}
The numerical semigroups $S = \<4,5,6,7\>$ and $S' = \<4,9,14,15\>$ lie in the same face of $P_4$ but the former has 5 Betti elements while the latter has 6.  Both~are maximal embedding dimension, and thus both have Kunz poset $P$ (depicted in Figure~\ref{sf:outerbettimed}) and Kunz nilsemigroup $N$.  One can check that $N$ has 6 outer Betti elements, namely
$$\begin{array}{l@{\qquad}l@{\qquad}l}
B_1 = \{(2,0,0)\}, 
&
B_2 = \{(0,2,0)\},
&
B_3 = \{(0,0,2)\},
\\
B_4 = \{(0,1,1)\}, 
&
B_5 = \{(1,0,1)\},
&
B_6 = \{(1,1,0)\}.
\end{array}$$
Since there is at least one outer Betti element in each equivalence class modulo 4, any numerical semigroup in the interior of $P_4$ must have at least 4 Betti elements.  Moreover, $\ol B_1 = \ol B_3$, but $(0,2,0,0)$ and $(0,0,0,2)$ cannot be factorizations of the same element of any numerical semigroup.
Lastly, $\ol B_2 = \ol B_5$, and the trades 
\[
(6,0,0,0) \sim (0,1,0,1)
\qquad \text{and} \qquad
(7,0,0,0) \sim (0,0,2,0)
\]
occur at the elements 24 and 28 in $S'$, respectively, while the trades 
\[
(3,0,0,0) \sim (0,1,0,1)
\qquad \text{and} \qquad
(3,0,0,0) \sim (0,0,2,0)
\]
both occur at the element 12 in $S$.  
\end{example}

\begin{lemma}\label{l:outerbettipartition}
Consider the graph with vertex set
\[
Z = \{z \in \mathsf Z_N(\infty) : z - e_i \notin \mathsf Z_N(\infty) \text{ for each } i \in \supp(z)\}
\]
such that two vertices $z, z' \in Z$ are connected by an edge whenever $z - e_i, z' - e_i \in \mathsf Z_N(p)$ for some $i \in \supp(z) \cap \supp(z')$ and $p \in N \setminus \{\infty\}$.  

\begin{enumerate}[(a)]
\item 
Each outer Betti element $B$ of $N$ is a connected component of $Z$, and the restriction of $Z$ to $B$ yields the graph $\nabla_B$.  

\item 
If $N$ is a Kunz nilsemigroup, then for each connected component $B$ of $Z$, we have $\ol z = \ol z'$ for all $z, z' \in B$.  

\end{enumerate}
\end{lemma}

\begin{proof}
It is clear any outer Betti element $B$ of $N$ is contained in $Z$ by Definition~\ref{d:outerbetti}(i), and that the restriction of the graph $Z$ to $B$ coincides with the graph $\nabla_B$.  Additionally, if $z \in B$ and $z' \in Z$ are connected by an edge, then there is some $i \in \supp(z) \cap \supp(z')$ with $z - e_i, z' - e_i \in \mathsf Z_N(p)$ for some $p \in N \setminus \{\infty\}$, so $\mathsf Z_N(p) + e_i \subset B$ and thus $z' \in B$.  This ensures $B$ equals the connected component of $Z$ containing it, proving part~(a).  

Next, by~transitivity and the connectivity of $B$, it suffices to assume that $z$ and $z'$ are connected by an edge, meaning there exists $i \in \supp(z) \cap \supp(z')$.  Then $B - e_i = \mathsf Z_N(p)$ for some $p \in N \setminus \{\infty\}$, so both $z - e_i$ and $z' - e_i$ are factorizations of $p$. Thus,
\[
\ol z = \ol{z - e_i} + p_i = \ol{z' - e_i} + p_i  = \ol z',
\]
thereby completing the proof.  
\end{proof}

\begin{example}\label{e:outerbettipartition}
The graph $Z$ in Lemma~\ref{l:outerbettipartition} can in general have connected components that are not outer Betti elements of $N$, even if $N$ is a Kunz nilsemigroup.  
Indeed, let $N$ denote the Kunz nilsemigroup of $S = \<11, 60, 72, 84, 96, 108\>$, depicted in Figure~\ref{sf:outerbettipartition}.  Clearly $(2,0,1,0,0) \in Z$, since $(1,0,1,0,0) \in \mathsf Z_N(1)$ and $(2,0,0,0,0) \in \mathsf Z_N(10)$, but
$$\mathsf Z_N(1) = \{(1,0,1,0,0), (0,2,0,0,0)\},$$
and $(1,2,0,0,0)$ lies outside of $Z$.  As such, $\mathsf Z_N(1) + e_1$ is not contained in any outer Betti element of $N$, and so $(2,0,1,0,0)$ cannot lie in any outer Betti element of $N$.  
\end{example}

\begin{remark}\label{r:computeouterbetti}
Lemma~\ref{l:outerbettipartition}(a) yields an algorithm to compute the outer Betti elements of any finite partly cancellative nilsemigroup.  Begin by computing the set $Z$ of minimal elements of $\mathsf Z_N(\infty)$ under the component-wise partial order.  
Next, compute the connected components of the graph in Lemma~\ref{l:outerbettipartition}, which has vertex set $Z$.  Each connected component $B$ is an outer Betti element if and only if Definition~\ref{d:outerbetti}(i) is satisfied.  
\end{remark}

In the rest of this section, assume $S = \<m, n_1, \ldots, n_k\>$ is a numerical semigroup with Kunz poset $P$ and Kunz nilsemigroup $N$, and let $p_i = \ol n_i \in \ZZ_m$ for $i = 1, \ldots, k$.  

\begin{defn}\label{d:m-centric}
Fix a minimal presentation $\rho$ of $S$.  A $(k+1)$-tuple $(z_0, z_1, \ldots, z_k)$ is said to be \emph{$m$-free} if $z_0 = 0$.  We say that a relation $(z,z') \in \rho$ is \emph{$m$-centric} if $\wh z \in \mathsf Z_P(\ol z)$, and that $\rho$ is \emph{$m$-centric} if every relation $(z,z') \in \rho$ is $m$-centric.  In particular, if $(z,z')$ is $m$-centric, then $z'$ must be $m$-free, and $z$ is $m$-free if and only if $\varphi_S(z) \in \Ap(S; m)$.  
\end{defn}

\begin{example}\label{e:m-centric}
Suppose $S = \<4,9,14,15\>$, and consider the factorization sets
\begin{align*}
\mathsf Z_S(24)
&= \{(6,0,0,0), (0,1,0,1)\}
\quad \text{and} \quad
\mathsf Z_S(28)
= \{(7,0,0,0), (1,1,0,1), (0,0,2,0)\}.
\end{align*}
Two outer Betti elements of the Kunz poset $P$ are $B_1 = \{(1,0,1)\}$ and $B_2 = \{(0,2,0)\}$.  Under Definition~\ref{d:m-centric}, the relation $((1,1,0,1), (0,0,2,0))$ is not $m$-centric, since 
\[
\ol{(1,1,0,1)} = \ol{(0,0,2,0)} = 0,
\]
but neither $(1,0,1)$ nor $(0,2,0)$ lie in $\mathsf Z_P(0)$.  In fact, any $m$-centric minimal presentation of $S$ must contain the relation $(z,z') = ((7,0,0,0), (0,0,2,0))$ since $\mathsf Z_P(0) = \{(0,0,0)\}$ (and, as we will see in Lemma~\ref{l:m-centricexists}(b), since 7 is the largest first coordinate of any factorization of~28).  
\end{example}



\begin{lemma}\label{l:m-centricexists}
\text{ }
\begin{enumerate}[(a)]
\item 
Every numerical semigroup has an $m$-centric minimal presentation.  

\item 
In any $m$-centric minimal presentation $\rho$ of $S$, each $(z,z') \in \rho$ satisfies \[
\varphi_S(z) - z_0 m \in \Ap(S;m).
\]
In particular, $z_0 >0$ if and only if $b = \varphi_S(z) \notin \Ap(S;m)$. 

\end{enumerate}
\end{lemma}

\begin{proof}
Fix a minimal presentation $\rho$ of $S$.  Any relation in $\rho$ occurring within the Ap\'ery poset will trivially be $m$-centric, as neither factorization involves the multiplicity.  Now, consider a Betti element $b \notin \Ap(S;m)$. By construction, $b = p + am$ for some $p \in \Ap(S;m) $ and $a > 0$, so there exists a factorization $z$ of $b$ with $z = ae_0 + z''$ for some $z'' \in \mathsf Z_S(p)$.  Any relation of the form $(z,z')$ for some $z'$ in a different connected component than $z$ is $m$-centric, so replacing each relation in $\rho$ occurring at $b$ with such a relation yields a minimal presentation by the paragraph after Theorem 10 in \cite{numericalappl}, one that is $m$-centric by construction.  This proves part~(a).  

Next, suppose $\rho$ is $m$-centric, fix $(z,z') \in \rho$, let $b = \varphi_S(z) = \varphi_S(z')$ and $a \in \Ap(S;m)$ with $a \equiv b \bmod m$.  Since $\wh z \in \mathsf Z_P(\ol z)$, we have $(0, z_1, \ldots, z_k) \in \mathsf Z_S(a)$, meaning 
\[
b = z_0m + \varphi_S(0, z_1, \ldots, z_k) = z_0m + a
\]
and thus $b - z_0m = a \in \Ap(S;m)$.  As such, $b \notin \Ap(S; m)$ if and only if $z_0 > 0$.  
\end{proof}

The following strengthens \cite[Proposition~8.19]{numerical} and has a nearly identical proof.  

\begin{lemma}\label{l:bettiapery}
Fix an $m$-centric minimal presentation $\rho$ of $S$, fix $(z,z') \in \rho$, and let $b = \varphi_S(z)$.  For each $i \in \supp(z')$, we have $b - n_i \in \Ap(S;m)$.
\end{lemma}

\begin{proof}
If $z_0 = 0$, then $b \in \Ap(S;m)$ by Lemma~\ref{l:m-centricexists}(b), so the claim immediately follows.  
On the other hand, if $z_0 > 0$, then since $z$ and $z'$ lie in different connected components of $\nabla_b$, we have $b - m - n_i \notin S$, meaning $b - n_i \in \Ap(S;m)$.  
\end{proof}

\begin{thm}\label{t:m-centric}
If $S$ has Kunz poset $P$ and $\rho$ is an $m$-centric minimal presentation of $S$, then 
\begin{enumerate}[(a)]
\item 
the set $\rho' = \{(\wh z, \wh z') : (z,z') \in \rho \text{ and } z_0 = 0\}$ is a minimal presentation for $P$, and

\item 
the set $\{ \wh z' : (z,z') \in \rho \text{ and } z_0 > 0\}$ consists of exactly one factorization from each outer Betti element of $P$.
\end{enumerate}
\end{thm}

\begin{proof}
For each $(z,z') \in \rho$ with $z_0 = 0$, since $\rho$ is $m$-centric Lemma~\ref{l:m-centricexists} implies $b = \rho_S(z) \in \Ap(S;m)$, so $z_0' = 0$ as well.  Thus $\wh{\mathsf Z_S (b)} = \mathsf Z_P(\ol b)$ by Theorem~\ref{t:posettosemigroupminpres}.  As~such, $\rho'$~is a minimal presentation for $P$ by Theorem~\ref{t:posettosemigroupminpres}, proving part~(a).  

Next, fix $(z,z') \in \rho$ with $z_0 > 0$ with Betti element $b = \varphi_S(z)$, and let
\[
B = \{z'' : z'' \text{ is in the same connected component of $\nabla_b$ as $z'$}\}.
\]
We wish to show $\wh{B} = \{\wh{z''}: z'' \in B\}$ is an outer Betti element of $P$.  Since $z$ and $z'$ have disjoint support, $z_0'' = 0$ for every $z'' \in B$, which implies $\wh{z}'' \in \mathsf Z_P(\infty)$.  By Lemma~\ref{l:bettiapery}, for each $i \in \supp(B)$, we have $b - n_i = a_j$ for some $a_j \in \Ap(S;m)$, and consequently $B - e_i \subseteq \mathsf{Z}(a_j)$.  For the other inclusion, $e_i + \mathsf Z_S(a_j) \subset \mathsf Z_S(b)$, and all of these factorizations are in the same connected component (as they are all supported on $e_i$).  Therefore, $B - e_i = \mathsf{Z}(a_j)$.  By Theorem~\ref{t:posettosemigroupminpres}, we have $\wh B - e_i =  \mathsf{Z}_P(p)$ for $p \in P$. Clearly $\nabla_{\wh B}$ is connected since every factorization is obtained from the same connected component of $\nabla_b$. Thus, $\wh B$ is an outer Betti element of $P$.

It remains to show that every outer Betti element of $P$ arises as a set $\wh B$ in the preceeding paragraph.  Fix $B \subset \{0\} \times \NN^k $ such that $\wh B$ is an outer Betti element of~$P$.  Since $\nabla_{\wh B}$ is connected, $B \subseteq \mathsf Z(b)$ for some $b \in S$ by Lemma~\ref{l:outerbettipartition}(b).  Since $\wh B \subseteq \mathsf Z_P(\infty)$, $b$ has at least one factorization $z$ with $z_0 > 0$.  Fix $i \in \supp (B)$.  Since $\wh B - e_i = \mathsf Z_P(p)$ for some $p \in P$, the bijection between $\mathsf Z_P(p)$ and $\mathsf Z_S(b - n_i)$ in Theorem~\ref{t:posettosemigroupminpres} guarantees any factorization of~$b$ supported on $i$ lies in $B$ and has first coordinate 0.  As such, $z$~lies in a different connected component of $\nabla_b$ than $B$, meaning $\nabla_b$ is disconnected and $b$ is a Betti element of~$S$.  
\end{proof}

\begin{cor}\label{c:minpressize}
Fix two numerical semigroups $S$ and $S'$ in some face $F \subset P_m$.  If $\rho$ and $\rho'$ are minimal presentations of $S$ and $S'$, respectively, then $|\rho| = |\rho'|$.  
\end{cor}

\begin{proof}
By Theorems~\ref{t:kunzlatticepts} and~\ref{t:groupconefacelattice}(b), $S$ and $S'$ have identical Kunz poset, so apply Theorem~\ref{t:m-centric}.  
\end{proof}


When presented with a face $F \subseteq P_m$, the following corollary of Theorem~\ref{t:m-centric} gives rise simultaneously to a minimal presentation of every numerical semigroup on $F$.  We first see an example.  

\begin{example}\label{e:kunzcoordformula}
Consider the Kunz poset $P$ from Example~\ref{e:outerbettimerge}.  The given minimal presentation for $S = \<9, 20, 30, 35\>$ includes the trade $(5,0,1,0) \sim (0,2,0,1)$, but in fact, any numerical semigroup $S'$ with multiplicity $m = 9$ and Kunz poset $P$ has a trade of the form $(\ell, 0, 1, 0) \sim (0, 2, 0, 1)$ for some $\ell$.  Moreover, if $S'$ has Kunz coordinates $x \in P_9$, 
then 
\[
\ell \cdot 9 + 1 \cdot (9x_3 + 3) = 2 \cdot (9x_2 + 2) + 1 \cdot (9x_8 + 8)
\]
and thus $\ell = 1 + 2x_2 - x_3 + x_8$.  By Theorem~\ref{t:m-centric}, we obtain
\[
\begin{array}{r@{\qquad}r@{\qquad}r}
(0,3,0,0) \sim (0,0,2,0),
&
(\ell_1,1,0,0) \sim (0,0,1,1),
&
(\ell_3,0,1,0) \sim (0,2,0,1),
\\
(0,2,1,0) \sim (0,0,0,2),
&
(\ell_2,0,0,1) \sim (0,1,2,0),
&
(\ell_4,0,0,0) \sim (0,3,1,0)\phantom{,}
\end{array}
\]
as an $m$-centric minimal presentation for $S'$, where
$$
\ell_1 = 1 - x_2 + x_3 + x_8,
\quad
\ell_2 = x_2 + 2x_3 - x_8,
\quad
\ell_3 = 1 + 2x_2 - x_3 + x_8,
\quad
\ell_4 = 1 + 3x_2 + x_3
$$
are affine linear combinations of the Kunz coordinates of $S'$.  
\end{example}

\begin{cor}\label{c:kunzcoordformula}
Let $L$ denote the set of affine linear functions $\ell:\RR^{m-1} \to \RR$.  Fix a face $F \subseteq P_m$ containing numerical semigroups, and let $p_1, \ldots, p_k \in \ZZ_m$ be the atoms of the Kunz poset $P$ associated to~$F$.  There exists a set of functions
\[
\rho_F \subseteq (L \times \NN^k) \times (\{0\} \times \NN^k),
\]
on $\RR^{m-1}$ such that
for every numerical semigroup $S$ in $F$, the evaluation 
\[
\rho_F(x) = \{((\ell(x), z), (0, z')) : ((\ell, z), (0, z')) \in \rho_F\}
\]
at the Kunz tuple $x$ of $S$ is a minimal presentation of $S$. 
\end{cor}

\begin{proof}
Fix a minimal presentation~$\rho'$ of $P$.  For each outer Betti element $B$ of $P$, choose $z' \in B$, choose $z \in \mathsf{Z}_P(\ol{z}')$, define a function $\ell:\RR^{m-1} \to \RR$ by
$$\ell(x) = \sum_{i=1}^{k} z'_i\big(x_{p_i} + \frac{p_i}{m}\big) - \sum_{i=1}^k z_i\big(x_{p_i} + \frac{p_i}{m}\big)$$
and define $r_B = ((\ell,z), (0,z'))$.  Finally, let 
$$\rho_F = \{((0,z),(0,z')): (z,z') \in \rho' \} \cup \{r_B : B \text{ is an outer Betti elements of $P$}\}.$$
By Theorems~\ref{t:allminimalpresentations} and~\ref{t:m-centric}, $\rho_F(x)$ is a minimal presentation of the numerical semigroup~$S$ with Kunz tuple $x$.  
\end{proof}

\begin{remark}\label{r:minpresalgorithm}
One consequence of Corollary~\ref{c:kunzcoordformula} is a purely combinatorial algorithm for computing the minimal presentation of a numerical semigroup, wherein outer Betti elements are computed with Lemma~\ref{l:outerbettipartition}.  Used in conjuction with \cite[Algorithm~7.1]{kunzfaces1}, a refinement of Wilf's circle of lights theorem~\cite{wilfconjecture}, Corollary~\ref{c:kunzcoordformula} yields an algorithm whose worst-case runtime depends only on the multiplicity, since there are only finitely many possible Kunz posets for each fixed $m$.  

We note that this algorithm itself is not new.  
The \texttt{GAP} package \texttt{numericalsgps}~\cite{numericalsgpsgap}, one of the primary packages for numerical semigroup computations, offers 2 primary methods for computing minimal presentations:\ one uses Gr\"obner bases, while the other uses the Ap\'ery set and factorization graphs~\cite{compoverview}.  The algorithm described here is effectively equivalent to the latter method, and one would expect an implementation to have comparable runtime.  However, the formulation here (i) identifies that the runtime can be bounded in terms of $m$, and (ii) depends only on the Kunz poset~$P$ of $S$, in that it simultaneouly obtains a minimal presentation for any numerical semigroup with Kunz poset $P$.  
It would be interesting to locate the optimal big-$O$ for this algorithm, and determine which Kunz posets demand maximal runtime.  
%
\end{remark}


\begin{remark}\label{r:minpressizem4m5}
Fix a numerical semigroup $S$ with multiplicity $m$, and fix a minimal presentation $\rho$ for $S$.  It is known \cite[Corollary~8.27]{numerical} that $|\rho| \le \binom{m}{2}$, with equality if and only if $\mathsf e(S) = m$.  With Theorem~\ref{d:m-centric}, we can obtain a full list of the possible values of $|\rho|$, and we immediately see that not all cardinalities strictly less than $\binom{m}{2}$ are possible.  For example, if $m = 4$, then $|\rho| \in \{1, 2, 3, 6\}$, and if $m = 5$, then $|\rho| \in \{1, 2, 3, 5, 6, 10\}$.  It remains an interesting question to determine which cardinalities are missing.  Aside from some progress in~\cite{highembdim}, Question~\ref{q:minpressizes} remains wide open in general.  
\end{remark}

\begin{question}\label{q:minpressizes}
Given a multiplicity $m$, what are the possible cardinalities of a minimal presentation of a numerical semigroup $S$ with multiplicity $m$?  
\end{question}

\section{The \texttt{Sage} package \texttt{KunzPoset}}
\label{sec:sagepackage}

Throughout the present project, we developed a \texttt{Sage} package \texttt{KunzPoset} that provides an abstract class for Kunz posets~\cite{kunzposetsage}.  Much of the functionality discussed in this paper and its predecessors \cite{kunzfaces2,kunzfaces1} has been implemented, including factorizations and minimal presentations, as well as more general Kunz polyhedron functionality, such as locating semigroups on a given face (equivalently, with a given Kunz poset).  

The \texttt{KunzPoset} object can be constructed from numerous types of input data, including a hyperplane description of the corresponding face, a numerical semigroup with that Kunz poset (specified via a list of generators, Kunz coordinates, or an Apery set, or as an instance of the \texttt{NumericalSemigroup} class), or a poset (as a list of cover relations or an existing \texttt{FinitePoset} object).  
Some of these require the user to install the \texttt{GAP} package \texttt{numericalsgps}~\cite{numericalsgpsgap} as well as the \texttt{NumericalSemigroup.sage} package~\cite{numericalsgpssage}.
It is also possible to enumerate all Kunz posets of a given Kunz polyhedron~$P_m$, using face lattice output from the \texttt{normaliz} package~\cite{normaliz3}.  

We close by providing a brief snippet of sample code.  
The poset constructed therein appeared in Examples~\ref{e:groupconefacelattice},~\ref{e:minpres} and~\ref{e:posetdimension}, and the output is as expected.  
For complete documentation, we direct the reader to the package source, available at
\begin{center}
\url{https://github.com/coneill-math/kunzpolyhedron}
\end{center}
with hosting provided by Github. 
\begin{verbatim}
sage: load("./KunzPoset.sage")
sage: import pprint  # Only necessary for printing purposes
sage: hplanes = [[ 0, 0, 2, 0, 0,-1, 0], [ 0,-1, 1, 0, 0, 0, 1],
....:            [-1, 0, 0, 1, 1, 0, 0], [ 0, 0, 0, 0, 0,-1, 2],
....:            [-1, 0, 1, 0, 0, 1, 0], [-1, 1, 0, 0, 0, 0, 1]]
sage: P = KunzPoset(m=8, hyperplane_desc=hplanes)
sage: pprint.pprint(P.Factorization())
{0: [[0, 0, 0, 0]],
 1: [[3, 0, 0, 0], [1, 0, 0, 2], [0, 1, 1, 0]],
 2: [[1, 0, 0, 1]],
 3: [[1, 0, 0, 0]],
 4: [[0, 1, 0, 0]],
 5: [[0, 0, 1, 0]],
 6: [[2, 0, 0, 0], [0, 0, 0, 2]],
 7: [[0, 0, 0, 1]]}
sage: print(P.BettiMatrix())
[ 2  0  0 -2]
[ 3 -1 -1  0]
sage: print(P.Dimension())
2
\end{verbatim}


\end{document}